\documentclass[reqno]{amsart}
\usepackage{epsfig,latexsym,amsfonts,amssymb,amsmath,amscd}
\usepackage{amsthm}

\usepackage{tikz-cd}
\usepackage[mathscr]{eucal}
\usepackage{mathrsfs}
\usepackage{mathbbol}
\usepackage{array}
\usetikzlibrary{matrix,arrows,decorations.pathmorphing}
\usepackage{oldgerm,units,color}
\usepackage{hyperref}

\usepackage{pst-node}
\usepackage{graphicx}

\usepackage{eepic, epic}

\usepackage{ifthen}

\def\ssemiring0{$s$-semiring$^\dagger$}

\usepackage{tikz}
\usepackage{epsfig,latexsym,amsfonts,amssymb,amsmath,amscd,graphics,epic}
\usepackage{amsfonts,amssymb,amsmath,amscd,amsthm}
\usepackage[mathscr]{eucal}
\usepackage{mathrsfs}

\usepackage{mathbbol}

\usepackage{oldgerm,units}
\usepackage{wrapfig,epsfig}

\usepackage{ifthen}
\usepackage{mathbbol}

\usepackage{amsthm}
\usepackage[mathscr]{eucal}
\usepackage{mathrsfs}
\usepackage{mathbbol}
\usepackage{oldgerm,units}
\usepackage{wrapfig}

\newtheorem{theorem}{Theorem}[section]

\newtheorem{example}[theorem]{Example}
\newtheorem{remark}[theorem]{Remark}

\newcommand{\Net}{\mathbb N}

\newcommand{\one}{\mathbb{1}}
\newcommand{\zero}{\mathbb{0}}

\newcommand{\trop}[1]{\mathcal{#1}}

\newcommand{\tG}{\trop{G}}

\newcommand{\tT}{\trop{T}}

\newcommand{\Hom}{Hom}

    \ifx\proof\undefined
    \newenvironment{proof}{
    \smallskip
    \noindent\emph{Proof.}}{\hfill\(\Box\)
    \bigskip
    } \fi

\newcommand{\ifdef}[3]{\ifthenelse{\equal{#1}{true}}{#2}{#3}}

\definecolor{lgray}{gray}{0.90}

\def\Mod{\operatorname{Mod}}

\def\ctw{\cdot_{\operatorname{tw}}}

\def\vep{\varepsilon}

\def\({\left(}
\def\){\right)}

\def\pipe{{\underset{{\ \, }}{\mid}}}

\def\vsemifield0{$\nu$-semifield$^\dagger$}
\def\vsemiring0{$\nu$-semiring$^\dagger$}

\def\pipe1{{\underset{{1}}{\mid}}}
\def\lmod1{\mathrel  \pipe1  \joinrel \joinrel =}

\def\CFunFF1{\operatorname{CFun} (F,F)}

\def\semiring0{semiring$^{\dagger}$}
\def\Semiring0{Semiring$^{\dagger}$}
\def\Semirings0{Semirings$^{\dagger}$}
\def\semidomain0{semidomain$^{\dagger}$}
\def\semifield0{semifield$^{\dagger}$}
\def\semifields0{semifields$^{\dagger}$}
\def\vsemifields0{$\nu$-semifields$^{\dagger}$}
\def\domain0{domain$^{\dagger}$}
\def\predomain0{pre-domain$^{\dagger}$}
\def\predomains0{pre-domains$^{\dagger}$}
\def\domains0{domains$^{\dagger}$}
\def\vdomains0{$\nu$-domains$^{\dagger}$}

\def\domains0{domains$^\dagger$}

\def\ker{\operatorname{ker}}

\newcommand{\etype}[1]{\renewcommand{\labelenumi}{(#1{enumi})}}
\def\eroman{\etype{\roman}}

\def\pipe{{\underset{{\tG}}{\mid}}}

\def\lmod{\mathrel  \pipe \joinrel \joinrel =}
\def\pipe{{\underset{{\tG}}{\mid}}}

\def\diag{\operatorname{diag}}

\newtheorem{thm}[theorem]{Theorem}

\newtheorem*{thm*}{Theorem}

\newtheorem{cor}[theorem]{Corollary}

\def\Hom{\operatorname{Hom}}
\def\Mor{\operatorname{Mor}}
\newtheorem{lem}[theorem]{Lemma}
\newtheorem{rem}[theorem]{Remark}
\newtheorem{prop*}{Proposition}

\newtheorem{prop}[theorem]{Proposition}
\newtheorem{defn}[theorem]{Definition}
\newtheorem*{examp*}{Example}
\newtheorem*{examples*}{Examples}
\newtheorem*{remark*}{Remark}
\newtheorem*{defn*}{Definition}

\newtheorem*{nothma}{\textbf{Theorem A}}
\newtheorem*{nothmb}{\textbf{Proposition B}}
\newtheorem*{nothmbp}{\textbf{Proposition B'}}
\newtheorem*{nothmd}{\textbf{Proposition D}}
\newtheorem*{nothmd'}{\textbf{Proposition D'}}
\newtheorem*{nothme}{\textbf{Theorem E}}
\newtheorem*{nothmf}{\textbf{Theorem F}}
\newtheorem*{nothmg}{\textbf{Theorem C}}

\def\tT{\mathcal T}

\def\tTz{\tT_\zero}

\numberwithin{equation}{section}

\def\M0{M_{\zero}}

\def\PS{P}

\def\semirings0{semirings$^\dagger$}

\newcommand{\nPS}[1]{\PS_{(!#1)}}
\newcommand{\nPSo}[1]{\nPS{\one}}

\begin{document}


\title[Projective systemic modules]
{Projective systemic modules}


\author[J.~Jun]{Jaiung~Jun}
\address{Department of Mathematics, University of Iowa, Iowa City, IA 52242, USA} \email{jujun0915@gmail.com}

\author[K.~Mincheva]{Kalina~Mincheva}
\address{Department of Mathematics, Yale University, New Haven, CT 06511, USA}
\email{kalina.mincheva@yale.edu}

\author[L.~Rowen]{Louis Rowen}
\address{Department of Mathematics, Bar-Ilan University, Ramat-Gan 52900, Israel} \email{rowen@math.biu.ac.il}




\thanks{\noindent \underline{\hskip 3cm } }



\begin{abstract}

We develop the basic theory of projective modules and splitting in
the more general setting of systems. Systems provide a common
language for most tropical algebraic approaches including
supertropical algebra, hyperrings (specifically hyperfields), and
fuzzy rings. This enables us to prove analogues of classical
theorems for tropical and hyperring theory in a unified way. In this
context we prove a Dual Basis Lemma and  versions of Schanuel's
Lemma. 
\end{abstract}
\maketitle





\section{Introduction}

\subsection{Motivation}$ $

In recent years, there has been a growing interest in developing
 theories of algebraic structures, more general than (commutative)
rings, such as semirings, hyperrings (specifically hyperfields),
fuzzy rings and supertropical algebra.
The main motivation for the study of these structures is the number of applications to other areas of mathematics.
For instance, semirings arise naturally in tropical
algebraic geometry \cite{MS}. Using hyperfields
(which generalize fields by allowing ``multi-valued'' addition)
Baker and Bowler in \cite{BB1} successfully unify
various generalizations of matroids (combinatorial abstraction of
vector spaces) in an elegant way. Baker and Bowler's work is based
on the interesting idea  (as initiated by the third author in
\cite{Row16} building on \cite{Ga} and \cite{Lor1},  and then
\cite{JuR1}),  that one is able to treat certain well-known (but not
directly related) mathematical structures simultaneously.
To this end, one needs to appeal to more general algebraic structures than commutative rings (cf. \cite{BL2}).

The idea of proving results about classical objects, using these
general algebraic structures, has already been implemented in the
literature. For example, in \cite{Jun} the first author proved that
several topological spaces (algebraic varieties, Berkovich
analytifications, and real schemes) can be seen as sets of
``$H$-rational points'' of algebraic varieties for some hyperfields
$H$. In \cite{AD}, Anderson and Davis defined the notion of
hyperfield Grassmannians, generalizing a MacPhersonian (certain
moduli space of matroids). Furthermore, in a recent paper \cite{BL},
Baker and Lorscheid proved that certain moduli functors (of
matroids) are representable by algebraic structures called
pasteurized ordered blue prints.

In this paper, we continue developing the theory of the common framework for
the generalized algebra structures, called a \textbf{system}. As indicated in
Example~\ref{precmain}, this ``systemic'' theory encompasses most
algebraic approaches to tropical mathematics.
 In Examples~\ref{morphmean} and \ref{morphm}
we state explicitly for the reader's convenience how the systemic (generalized)
version of morphisms (which we call $\preceq$-morphisms) translates to tropical mathematics,
hyperrings, and fuzzy rings.

In \cite{IKR7} the famous basic structure theorems about composition
series, noetherian and artinian properties, etc., were studied for a
specific class of modules, called SA-modules. These arise in
tropical algebra, but not in classical algebra. Here we take a more categorical perspective,
and projective modules play the major role. Projective modules over semirings,
whose theory is analogous to classical exact sequences and module theory,
appear in \cite[Chapter~17]{golan92} and \cite{Ka1}, and have been
studied rather intensively over the years,
\cite{DeP,IJK,IKR6,Ka3,KN,Mac}.

An equivalent definition of projective module in   classical ring
theory is as a direct summand of a free module, but over an
arbitrary semiring this property is considerably stronger. (See
\cite[Example~4.6]{IKR6} for a projective module over a semiring
which is not a summand of a free module). The strong decomposition
results given in \cite{IKR6}
 rely on this more restrictive definition and show that all indecomposable
 ``strongly''
 projective modules over a ring ``lacking zero sums'' are
 principal, and thus the Grothendieck group is trivial. This view is continued in~\cite{KNZ}.
However, direct sums are ``too good'' to lead to a viable homology theory
over semirings.

In this paper we  return to the general categorical definition of
projective, cast in the language of systems, a crucial feature of
which is the ``surpassing relation'' $\preceq$ which generalizes
equality and is needed to reformulate analogs of   classical
ring-theoretic theorems in the semiring context. We  consider   more
general notions, $\preceq$-projectivity and $\succeq$-projectivity,
given in Definition~\ref{precproj},  based on
``$\preceq$-splitting'' and ``$\succeq$-splitting''
in~Definition~\ref{precspli}.  
Bringing  $\preceq$ into the picture,   has the following advantages:

 \begin{enumerate} \eroman
\item $\preceq$-morphisms include tropicalization, as indicated in
Example~\ref{morphmean}(i).
\item Idempotent mathematics can be formulated as a special case of
the symmetrized semiring and module.
\item Applications include modules over hyperfields.
\item The class of projectives is broadened to include  Example~\ref{idemproj}.
\item Classical results about projective modules (their characterization,  the Dual
Basis Lemma, and Schanuel's Lemma) are  a special case of the
$\preceq$-version, many of which are not available without the use
of $\preceq$.
\item  One can continue in analogy with classical lines, such as
Morita theory, which already has been treated in \cite{KN} and
\cite{Row19}. This thread is continued in homological algebra
\cite{JMR1} and other work in progress.
\end{enumerate}
Along the way, a more appropriate (and more general)
$\preceq$-version and $\succeq$-version of direct sum is given in
Definition~\ref{definition: direct sum}, via  systemic
generalization of splitting in \S\ref{splitont} and
characterizations of $\preceq$-projectivity in~\S\ref{prsys}.


\subsection{Main results}

\begin{nothma}[Theorem~\ref{splitdir}]
Let $\pi: \mathcal M \to \mathcal N$ be a homomorphism. If $ \nu $
$\preceq$-splits $\pi$ , then:
\begin{enumerate}\eroman
    \item
    $\mathcal M$ is
    the $\preceq$-direct sum  of $\mathcal M_1: =
    \pi(\mathcal M)$ and $\mathcal M_2: = (\one_\mathcal M (-)\nu
    \pi)(\mathcal M)$ with respect to the $\preceq$-morphisms
     $\pi_1 =\pi, \, \nu _1  = \nu ,\,$ $\pi_2 =
    (\one _\mathcal M (-)\nu \pi),$ $ \nu_2 = \one_{\mathcal M_2}$.
    \item $\mathcal M$ is
    the $\preceq$-direct sum  of $\mathcal M_1= \nu
    \pi   (\mathcal M)$ and $\mathcal M_2 = \ker_{\Mod,\mathcal M}\pi$,
    with respect to $\nu _i = \one_{\mathcal M_i}$ for $ i = 1,2$.
\end{enumerate}
\end{nothma}

This statement holds for $h-$splitting as well.

\begin{nothmb}[Proposition~\ref{epicspl}]
The following are equivalent for a systemic module  $\mathcal P$:
\begin{enumerate} \eroman
    \item
    $\mathcal P$  is  $(\preceq,h)$-projective.
    \item
    Every   $\preceq$-onto homomorphism to   $\mathcal P$
    $\preceq$-splits.
    \item
    There is a $\preceq$-onto homomorphism from a free system to
    $\mathcal P$ that  $\preceq$-splits.
    \item
    Given a $\preceq$-onto $\preceq$-morphism $h:\mathcal M \to \mathcal
    M'$, the map $ \Mor_\preceq (\mathcal P,h): \Mor_\preceq(\mathcal P,
    \mathcal M) \to \Mor_\preceq(\mathcal P,\mathcal M')$ given by $ g
    \mapsto hg$ is $\preceq$-onto.
\end{enumerate}

The h-version also holds (Proposition~\ref{epicspl1}).
\end{nothmb}

\begin{nothmbp}[Proposition~\ref{epicspl10}] The following are equivalent for a systemic module  $\mathcal P$:
\begin{enumerate} \eroman
\item
$\mathcal P$  is   $\succeq$-projective.
\item Every $\succeq$-onto  $\preceq$-morphism to $\mathcal P$ $\succeq$-splits.
\item There is a $\succeq$-onto $\preceq$-morphism  from a free system to $\mathcal P$ that  h-$\succeq$-splits.
\item The functor $\Hom (\mathcal P,\underline{\phantom{w}})$
 sends  $\succeq$-onto $\preceq$-morphisms to
$\succeq$-onto $\preceq$-morphisms.
\end{enumerate}
\end{nothmbp}

In the following statements we use various notions of {\bf kernel}
for a map $f: \mathcal{M} \rightarrow \mathcal{N}$ (cf.
Definition~\ref{modker} and Definition~\ref{congker}):
\begin{enumerate}
    \item Null-module kernel $\ker_{\Mod,\mathcal M} f$ defined as the preimage of the set $\{a \in \mathcal{N}: a \succeq \zero\}.$
    \item Congruence kernel $\ker_N f:=\{ (a_0,a_1) \in \mathcal M \times \mathcal M : f(a_0) =f(a_1) \}.$
    \item $\preceq$-congruence kernel $\ker_{N,\preceq} f:=\{ (a_0,a_1) \in \mathcal M \times \mathcal M : f(a_0) = f(a_1) ,\quad f(a_0), f(a_1) \in \mathcal N_\textrm{Null} \}.$
\end{enumerate}

\begin{nothmg} [Theorem~\ref{Sch2}] If $\mathcal P_1 $ is
    $(\preceq,h)$-projective  with a  $\preceq$-onto homomorphism $\pi: \mathcal P
    \longrightarrow \mathcal P_1 $ whose null-module kernel $\mathcal K$
    is
    $(\preceq,h)$-projective, then $\mathcal P $  also is
    $(\preceq,h)$-projective.
\end{nothmg}

\begin{nothmd}[Proposition~\ref{Dualbas}---``$\preceq$-Dual Basis Lemma'']
A module pseudo-system $(\mathcal P, \tT_\mathcal P,(-),\preceq)$
that is $\preceq$-generated by $\{p_i \in \mathcal P:i \in I\}$ is
$(\preceq,h)$-projective (resp.~h-projective) if and only if there
are $\preceq$-onto $\preceq$-morphisms (resp.~homomorphisms) $g_i :
\mathcal P \to \mathcal A$ such that for all $a\in \mathcal A$ we
have $a \preceq\sum g_i(a) p_i,$ where $g_i(a) = \zero$ for all but
finitely many $i$.
\end{nothmd}

\begin{nothmd'}[Proposition~\ref{Dualbas1}---``$\succeq$-Dual Basis Lemma'']
Suppose a module pseudo-system $(\mathcal P, \tT_\mathcal
P,(-),\preceq)$ is generated by $\{p_i \in \mathcal P:i \in I\}$.
Then $\mathcal P$ is $(\succeq,h)$-projective if and only if there
are $\succeq$-onto $\succeq$-morphisms $g_i : \mathcal P \to
\mathcal A$ such that for all $a\in \mathcal A$ we have $a
\succeq\sum g_i(a) p_i,$ where $g_i(a) = \zero$ for all but finitely
many $i$.
\end{nothmd'}

 These are tied in
with $\preceq$-idempotent and $\preceq$-von Neumann regular matrices
in Proposition~\ref{kertriv1} and Corollary \ref{vNr}.

 With the basic definitions and properties in hand, one is
ready to embark on the part of module theory involving projective
 modules. Our main   application is  Schanuel's Lemma over
 semirings.

\begin{nothme}[Theorem~\ref{trSh}]
   Suppose we have $\preceq$-morphisms $\mathcal P_1  \overset {f_1} \longrightarrow \mathcal M$ and
    $\mathcal P_2  \overset {f_2} \longrightarrow \mathcal M$  with $f_1$ and $f_2$  onto. (We are
    not assuming that either $\mathcal P_i $ is projective.) Let
 $$\mathcal P = \{ (b_1, b_2)   :\quad b_i \in \mathcal P_i, \ f_1(b_1) =
        f_2(b_2)\},$$   a submodule of $\mathcal P_1 \oplus \mathcal P_2$,
        together the restriction $\pi_i^{\operatorname{res}}$ of
        the projection  $\pi_i: \mathcal P \to \mathcal P_i$ on the $i$ coordinate, for $i = 1,2$.
    \begin{enumerate}\eroman
        \item   $\pi_1^{\operatorname{res}}:\mathcal P \to \mathcal P_{1}$ is an onto homomorphism and
        and there is an
        onto homomorphism $$ \ker_N \pi_1^{\operatorname{res}} \to \ker_N  f_2, $$
        (This part is
        purely semiring-theoretic and does not require a system.)
            \item The maps $f_1\pi_1^{\operatorname{res}} ,f _2 \pi_2^{res}: \mathcal P \to \mathcal M$ are the same.
     \item
        In the systemic setting, $\pi_1^{\operatorname{res}}$ also induces
         $\preceq$-quasi-isomorphism
     \[
    \pi_{N,\preceq}:  \ker_{N, \preceq} \pi_1^{\operatorname{res}} \to \ker_{N, \preceq} f_2.\]

     \item In (iii), if $f_1$ also is null-monic, we have the following $\preceq$-quasi-isomorphism:
    \[
     \ker_{N, \preceq}  f_1\pi_1^{\operatorname{res}} \to \ker_{N, \preceq}  f_2.
     \]

 \item
 If $\mathcal P_1 $ is projective,   then it is a retract
 of $\mathcal P$ with respect to the projection $\pi_1:\mathcal P \to \mathcal
 P_1$.

 \item
 If $\mathcal P_1 $ is $\preceq$-projective, then it is a
 $\preceq$-retract
 of $\mathcal P$ with respect to the projection $\pi_1:\mathcal P \to \mathcal P_1$, and $\mathcal P$ is the $\preceq$-direct sum of
 $\mathcal P_1$ and $(\one_\mathcal P (-) \nu_1 \pi_1)(\mathcal P)$.

    \end{enumerate}
 \end{nothme}

We also have a $\preceq$-onto $\preceq$-version, given in
Lemma~\ref{trSh11}.

 \begin{nothmf}[Theorem~\ref{trSh119} (Semi-Schanuel,
    $\preceq$-version)]
    Given a $\preceq$-morphism $\mathcal P   \overset {f }
\longrightarrow \mathcal M$ and a homomorphism $\mathcal P'
\overset {f'} \longrightarrow \mathcal M'$  with $\mathcal P$ and
$\mathcal P'$ $\preceq$-projective and $f$  $\preceq$-onto,
    and a $\preceq$-onto $\preceq$-morphism $\mu: \mathcal M \to \mathcal
    M'$,
let $\mathcal K =  \ker_{\Mod,\mathcal P}f$ and  $\mathcal K' =
\ker_{\Mod,\mathcal P'} f'.$ Then there is a $\preceq$-onto
$\preceq$-splitting $\preceq$-morphism $g: \mathcal K' \oplus
\mathcal P \to \mathcal P'$, with a $\preceq$-morphism $\Phi:
\mathcal K \to \ker_{\Mod,\mathcal K' \oplus \mathcal P}g $ which is
1:1 (as a set-map).
\end{nothmf}

This could be used in conjunction with Theorem~\ref{Sch2}.

Our approach throughout this paper is explicit, aimed to show how
projective systemic modules work, especially since subtle variations
of the definitions lead to differing results. The category of
$\tT$-modules has enough $(\preceq,h)$-projectives and
$h-$projectives (because every free module is also projective,
$(\preceq,h)$-projective, and h-projective) to define
$\preceq$-projective resolutions, but despite Theorem~\ref{Sch2} we
do not yet have decisive enough results along these lines to include
here (although Corollary~\ref{Sch29} indicates how the theory might
develop).
%
%
%
%

\par\medskip
\textbf{Acknowledgments} J.J. was supported by AMS-Simons travel grant. K.M. was supported by the Institute Mittag-Leffler and the ``Vergstiftelsen''. K.M. would like to thank the Institute Mittag-Leffler for its hospitality. Part of this work has been carried out during the workshop ``Workshop on Tropical varieties and amoebas in higher dimension'' in which K.M and L.R. participated.

%
%
%
%
%
%
%
%
%
%
%
%

\section{Basic notions}

 Throughout the
paper, we let $\mathbb{N}$ be the additive monoid of the non-negative integers. Similarly, we view $\mathbb{Q}$ (resp.~$\mathbb{R}$) as the additive monoid of the rational numbers (resp.~of the real numbers).

A \textbf{semiring}  $(\mathcal A, +, \cdot, 1)$ is an additive
commutative monoid $(\mathcal A, +, \zero)$ and multiplicative
monoid $(\mathcal A, \cdot, \one)$ satisfying the usual distributive
laws, cf.~\cite{golan92}.

\begin{rem}
Strictly speaking the element $\zero$ is not needed in semiring
theory, and one can make do later by adjoining the absorbing
element~$\zero$, but for convenience we will work with semirings and
assume $\zero \in \mathcal A.$
\end{rem}

We review the basic definitions and properties of $\tT$-modules,
triples,  and systems from \cite{Row17}; more details are given in
\cite{JuR1} and \cite{Row16}.

\begin{defn}\label{modu2}   A $\tT$-\textbf{module} over a set $ \tT$
 is an additive monoid $( \mathcal A,+,\zero_\mathcal A)$ with a scalar
multiplication $\tT\times \mathcal A \to \mathcal A$ satisfying the
following axioms, $\forall u \in \Net,$ $a \in \tT,$ $b,b_j \in
\mathcal A$:

 \begin{enumerate}\label{distr3}\eroman
   \item (Distributivity over $\tT$): $a (\sum _{j=1}^u b_j) =     \sum _{j=1}^u (a b_j). $
\item  $a\zero_{\mathcal A }= \zero_{\mathcal A }$.
\end{enumerate}
\end{defn}



We review some definitions for convenience. We start off with  a
$\tT$-module $\mathcal A$, perhaps with extra structure. When $\tT$
is a monoid we call $\mathcal A$ a $\tT$-\textbf{monoid module}. We
can make $\mathcal A$ into a semiring by means of
\cite[Theorem~2.5]{Row16}, in which case we essentially have
Lorscheid's blueprints, \cite{Lor1,Lor2}.

\subsection{Negation maps}

We introduce some more structure.

\begin{defn}\label{negmap}
 A \textbf{negation map} on  a $\tT$-module $\mathcal A$
is a monoid isomorphism
$(-) :\mathcal A \to \mathcal A$ of order~$\le 2,$  written
$a\mapsto (-)a$, which also
 respects the $\tT$-action in the sense that
$$(-)(ab) = a ((-)b),$$ for $a \in \tT,$ $b \in \mathcal A.$ \end{defn}

Assortments of negation maps are given in \cite{GatR,JuR1,Row16}. We
also remark that when $\one \in \tT  \subseteq \mathcal A$,   the
negation map $(-)$ is given by $(-)b = ((-)\one)b$ for $b \in
\mathcal A.$

We write $ a (-)a $ for $a+ ((-)a)$, and $ a ^\circ$ for $ a (-)a$,
called a \textbf{quasi-zero}.

\begin{rem}
Any quasi-zero is fixed by a negation map since $ (-) a ^\circ= (-)a
+a  =   a ^\circ.$ On the other hand, when $\mathcal A$ is
idempotent (i.e., $a+a=a$ for any $a \in \mathcal A$), any element
$a \in \mathcal A$  fixed by a negation map is a quasi-zero
since $a=(-)a$ and hence $a^\circ=a(-)a=a+a=a$.
\end{rem}

 The set $\mathcal A
^\circ$ of quasi-zeros is a $\tT$-submodule of $\mathcal A $ that
plays an important role. When $\mathcal A$ is a semiring,  $\mathcal
 A^\circ$ is an ideal.

\begin{defn}\label{modu21}
A  \textbf{pseudo-triple} $(\mathcal A, \tT, (-))$ is a $\tT$-module
$\mathcal A$, with $\tT$  a distinguished subset of $\mathcal A$,
called the set of \textbf{tangible elements}, and a negation map
$(-)$ satisfying $(-)\tT = \tT.$
\end{defn}

In this paper, we replace $\tT$ by a subset $\tT_{\mathcal
A}\subseteq \mathcal
 A$. We write $\tTz$ for $\tT \cup
\{\zero\}.$

\begin{defn} A \textbf{triple}
$(\mathcal A, \tT, (-))$
 is a pseudo-triple,
for which $\tT \cap \mathcal A^\circ = \emptyset$ and $\tTz$
generates $(\mathcal A ,+).$
\end{defn}

\subsection{Symmetrization and idempotent mathematics}\label{symm}$ $

When a $\tT$-module $\mathcal A$ does not come equipped with a
negation map, there are two natural ways to impose a negation map:
(1) one may take the negation $(-)$ to be the identity (for
instance, this is done in supertropical algebra), or  (2) one may
supply a negation map by ``symmetrizing''~$\mathcal A$, in a
procedure similar to the Grothendieck group completion. For more
details, see \cite{Ga} and then \cite[\S 1.3]{JuR1}. Symmetrization
is an important tool for idempotent mathematics and the max-plus
algebra, and
 plays a central role in our subsequent work~\cite{JMR1}. We briefly recall the basic definitions for the reader. 

For any $\mathcal{T}$-module $\mathcal{A}$, we let
$\widehat{\mathcal A} = \mathcal A \oplus \mathcal A$ and $\widehat
{\tT} = (\tT\oplus \zero) \cup (\zero\oplus \tT)$. The main idea, as
in the case of the group completion, is to consider the formal
construction of negation, and impose a canonical $\widehat{\mathcal
T}$-module structure on $\widehat{\mathcal A}$ as follows.

\begin{defn}\label{wideh7}
The \textbf{twist action} on $\widehat{\mathcal A}$ over $\widehat
{\tT} $
 is defined as follows:
 \begin{equation}\label{twi} (a_0,a_1)\ctw (b_0,b_1) =
 (a_0b_0 + a_1 b_1, a_0 b_1 + a_1 b_0), a_i \in \mathcal T, b_i \in \mathcal A.  \end{equation}

The \textbf{symmetrization} of $\mathcal A$ is the
$\widehat{\mathcal T}$-module $\widehat{\mathcal A}$ with the twist
action \eqref{twi}. A negation map is defined by using the \textbf{switch map}:
 \[
(-): \widehat{\mathcal A}\longrightarrow \widehat{\mathcal A}, \quad (b_0,b_1)\mapsto (b_1,b_0).
  \]
When $\mathcal A$ is a semiring, the twist action gives a semiring structure on $\widehat{\mathcal A}$ (together with coordinate-wise addition).
\end{defn}

\begin{rem}\label{sw} The switch map $(-)$ on $\widehat{\mathcal A}$ is a
negation map, and the quasi-zeros all have the form $(b,b)$ since
$(b_0,b_1)(-)(b_0,b_1)= (b_0+b_1,b_0+b_1).$

Suppose now that $\tT \subseteq \mathcal A$. (One could even take
 $\tT = \mathcal A\setminus {\zero}$.) Then $\mathcal
T_{\widehat{\mathcal A}}\cap {\widehat{\mathcal A}}^\circ =
\emptyset,$ and  $(\widehat{\mathcal A}, \mathcal
T_{\widehat{\mathcal A}},(-))$ is a triple for any $\tT$-module
$(\mathcal A, \tT, (-))$.

 The map $(\mathcal A,\tT,(-))\to (\widehat{\mathcal
A},\widehat{\tT},(-))$ sending $a \mapsto (a,\zero)$ and $b \mapsto
(b,\zero)$ for $a \in \tT$, $b \in \mathcal A,$ is a homomorphism of
triples. When $\mathcal A$ is additively idempotent, so is
$\widehat{\mathcal A}$. In this way, we   embed idempotent
mathematics into the theory of triples.

Thinking of $(b_0,b_1)$ intuitively as $b_0 - b_1$, we see that
$(b_1,b_0)$ corresponds to $b_1-b_0 = -(b_0 - b_1).$
\end{rem}

\subsection{Systems}$ $

 We round out the structure  with a \textbf{surpassing relation}
$\preceq$ given in \cite[Definition~1.70]{Row16} and also  described
in  \cite[Definition~3.11]{JuR1}.

\begin{defn}\label{precedeq07}
A \textbf{surpassing relation} on a triple $(\mathcal A, \tT, (-))$,
denoted
  $\preceq $, is a partial pre-order satisfying the following, for elements of $\mathcal A$:

  \begin{enumerate}
 \eroman
  \item   $c^\circ \succeq \zero$ for any $c\in \mathcal A$.
    \item If $b_1 \preceq b_2$ then $(-)b_1 \preceq (-)b_2$.
  \item If $b_1 \preceq b_2$ and $b_1' \preceq b_2'$ for $i= 1,2$ then  $b_1 + b_1' \preceq b_2
   + b_2'.$
    \item   If  $a \in \tT$ and $b_1 \preceq b_2$ then $a b_1 \preceq ab_2.$
    \item   If  $a\preceq b $ for $a,b \in \tT,$ then $a =  b.$
   \end{enumerate}


A \textbf{$\tT$-surpassing relation} on a triple $\mathcal A$  is a
surpassing relation also satisfying the following, for elements of
$\mathcal A$: if $b  \preceq a $ for $a \in \tT$ and $b \in \mathcal
A$, then $b=a$.

\end{defn}

   \begin{lem}\label{circget7}
  If $b_1 + c^\circ = b$ for some $c\in \mathcal A$, then $b_1 \preceq b$.
\end{lem}
\begin{proof} Since $c^\circ \succeq \zero,$ we can apply   Definition~\ref{precedeq07} (iii).
\end{proof}

The main case is in \cite[Definition~1.70]{Row16},
\cite[Definition~2.17]{JuR1}, defined as follows:
 \begin{equation}\label{preci} a_1\preceq_\circ a_2\text{ if }a_2 = a_1+ b^\circ\text{ for
some } b \in \mathcal A,\end{equation} but we also could take
$\preceq$ to be set inclusion when $\mathcal{A}$ is obtained from
the power set of a hyperring. See \cite[\S 10]{JuR1}.

\begin{lem}\label{circget} If $a_1\preceq a_2$, then $a_2 (-) a_1  \succeq \zero$ and  $a_1 (-) a_2 \succeq
\zero.$
\end{lem}
\begin{proof} $a_2 (-) a_1  \succeq a_1 (-) a_1  \succeq \zero$, and
thus $a_1 (-) a_2  = (-) (a_2 (-) a_1 )\succeq \zero$.
\end{proof}

 \begin{defn} $S_1\preceq S_2 $ for $S_1, S_2 \subseteq \mathcal A $
 if for each $s \in S_1$ there is $s' \in S_2$ for which  $s \preceq
 s'$.
\end{defn}

 \begin{defn}\label{Tsyst} A \textbf{system} (resp.~\textbf{pseudo-system}) is a quadruple $(\mathcal A, \tT ,
(-), \preceq),$ where $\preceq$ is a surpassing relation  on the
triple
  (resp.~ pseudo-triple) $(\mathcal A, \tT_{\mathcal A}  , (-))$,
  which
 is \textbf{uniquely negated} in the sense that
for any $a \in \tT$,  there is a unique element $b$ of $\tT_\mathcal
A$ for which $\zero \preceq a+b$ (namely $b = (-)a$).


A $\tT$-\textbf{system} is a system for which $\preceq$ is a
$\tT$-surpassing relation.
\end{defn}

\begin{rem}\label{sysgene}
Pseudo-systems encompass classical algebra, when we take $(-)$ to be
the usual negation~$-$, and $\preceq$ to be equality. This
``explains'' the parallel between so many theorems of tropical
algebra and classical algebra.
\end{rem}
%
%
%
%

For a pseudo-system $(\mathcal A, \tT , (-), \preceq),$ we define
the important $\tT$-submodule ${\mathcal A}_{\operatorname{Null}} =
\{a \in \mathcal A : a \succeq \zero\}$ of~${\mathcal A}$
containing~$\mathcal A ^\circ$.

Then in parallel to \eqref{preci} we have
\begin{defn}\label{precnul}   $b
\preceq_{\operatorname{Null}} b'$ when $b + c = b'$ for some $c \in
\mathcal A_{\operatorname{Null}}.$
\end{defn}
%

There are two ways that we want to view triples and their systems.
The first is as the ground structure on which we build our module
theory, in analogy to the ground ring for classical linear algebra
or for affine algebraic geometry. We call this a \textbf{ground
system}. We call $\mathcal A$ a \textbf{semiring
    system} when $\mathcal A$ is a semiring.

The second way, which is the main direction taken in this paper, is
to fix a ground triple $(\mathcal A, \tT , (-))$, and take $\mathcal
A$-modules $\mathcal M$ together with a distinguished subset
$\tT_{\mathcal M}$ spanning $\mathcal M$ and satisfying $\tT
\tT_{\mathcal M}\subseteq \tT_{\mathcal M}.$  We also require   $
\mathcal M$  to satisfy $((-)a)m = (-)(am)$ for $a \in \mathcal A, \
m \in \mathcal M.$ Then we define the \textbf{systemic module}
$(\mathcal M, \tT_{\mathcal M}, (-),\preceq)$ on $\mathcal M,$ to
satisfy the axiom

\medskip

 $ a_1 b_1 \preceq a_2b_2$ whenever $a_1 \preceq a_2$ in
$\tT_{\mathcal M}$ and $b_1 \preceq
 b_2$ in $\mathcal M$.

\medskip

    Right systemic modules are defined analogously. The detailed
study of such modules was carried out in \cite{JuR1}. Although the
two theories (ground systems and systemic modules) start off the
same, they quickly diverge, just as in classical algebra.

\begin{example}\label{precmain}$ $ $ $ \begin{enumerate}\eroman\item
Given a triple $(\mathcal A, \tT  , (-))$, take the surpassing
relation $\preceq$ to be $\preceq_\circ$ of \eqref{preci}; then
${\mathcal A}_{\operatorname{Null}} = \mathcal A^\circ$.

\item The set-up
of supertropical mathematics \cite{zur05TropicalAlgebra,IR} is a
special case of (i), where $\mathcal A = \tT \cup \tG$ is the
supertropical semiring, $(-)$ is the identity, $\circ$ is the
``ghost map,''  $ \tG =\mathcal A^\circ  $, and $\preceq$ is ``ghost
surpasses''. Another way of saying this is that $a_0 + a_1 \in
\{a_0,a_1\}$ for $a_0 \ne a_1 \in \tT$, and $a_0 + a_0 = a_0
^\circ.$ Tropical mathematics is encoded in $\tG$, which (excluding
$\zero$) often is an ordered group, and can be viewed for example as
the target of the Puiseux valuation (tropicalization).

\item The
fuzzy ring of \cite{Dr} is a special case of (i). More details are
given in \cite{Row16} and \cite{AGR}.

\item The
symmetrized triple can be made into  a system as special case of
(i), which includes idempotent mathematics, as was  explained in
Remark~\ref{sw}.

\item In the hypergroup setting, as described in
\cite[Definition~3.47]{Row16}, $\tT$ is a given hypergroup,
${\mathcal A}$ is the subset of the power set $\mathcal P(\tT)$
generated by $\tT$, and $\preceq$ is set inclusion. We call this a
\textbf{hypersystem}. ${\mathcal
    A}_{\operatorname{Null}}$ consists of those sets containing $\zero$,
which is the set of hyperzeros     in the hypergroup literature.

\item Tracts, introduced recently in \cite{BB1}, are mostly special cases of systems, where $\tT$ is the given Abelian group $ G$, $\mathcal A
= \mathbb N[G],$ $\vep = (-) \one,$ and $N_G $ is ${\mathcal
A}_{\operatorname{Null}},$ usually taken to be $\mathcal A^\circ$.
\end{enumerate}
\end{example}

 Examples \ref{precmain} can be unified for systems by taking $\preceq$ to be $\preceq_{\operatorname{Null}}$ of Definition \ref{precnul}.  Clearly this includes Example \ref{precmain}(i), and it also
includes Example \ref{precmain}(v) since $c \in \mathcal
A_{\operatorname{Null}}$ iff $\zero \in c,$ which implies $b
\subseteq b+c,$ for   $b,c \subseteq \mathcal P(\tT)$. We will also
want a weaker version of generation, which comes up naturally and
also ties into hyperrings.

\begin{rem} In a semiring, one has
the Green  relation given by $a \le b$ iff $a+b = b$,
\cite[Example~2.60(i)]{Row16}. Conversely, any ordered monoid with
$\zero$ gives rise to an idempotent semiring by putting $a+b = b$
whenever $a\le b.$

 The only natural
negation map here would be the identity, and one gets a
pseudo-triple by taking $\tT$ to be a generating set of $\mathcal
A$. But every element $a = a+a$ is a quasi-zero, and ${\mathcal
A}_{\operatorname{Null}}= \mathcal A$, so this pseudo-system is not
a system,  and   one does not get much structure theory along the
lines of systems. This is remedied in Example~\ref{precmain}(ii), by
symmetrization.

In the spirit of systems, our semirings will rather be
 ``almost'' idempotent (specifically
``$(-)$-bipotent,'' cf.~\cite[Definition 2.27]{Row16} and
\cite[Definition 2.8]{JuR1}).
\end{rem}

\begin{lem}\label{hyp7} Hypersystems   $({\mathcal A}= \mathcal P(\tT), \tT,
(-),\subseteq) $ satisfy the following property:

\medskip
If $a \in \tT$ and $a+b \succeq \zero$ for $b\in \mathcal A,$ then
$(-)a \preceq b.$
\end{lem}
\begin{proof}
$\zero \preceq a+b$ means that there exists $q \in b$ such that $\zero \in a+q$ (as a hypergroup). In particular, $q=-a$ and hence $(-a) \in b$, or $(-a) \preceq b$ since $\preceq$ is just the set-inclusion $\subseteq$ in this case.

\end{proof}

\begin{defn}\label{modu218}
Let $(\mathcal A, \tT,(-),\preceq)$ be a pseudo-system.
\begin{enumerate}\eroman
\item
An element  $b \in \mathcal A$ is \textbf{$\preceq$-generated} by a
subset $ \mathcal A'$ of  $\mathcal A$ if there is a subset $\{ a_i
: 1 \le i \le t\} \subseteq \mathcal A'$ such that  $b \preceq \sum
_i a_i.$
\item
For subsets $\mathcal A'$ and $\mathcal A''$ of $\mathcal A$, we say
that $\mathcal A'$ \textbf{$\preceq$-generates} $\mathcal A''$ if
each element of $\mathcal A''$ is $\preceq$-generated by $\mathcal
A'$.
\end{enumerate}
\end{defn}

 The $\succeq$-analog is less interesting because of the following
 reduction to usual generation.

\begin{lem}\label{monep7}
In a $\tT$-system $\mathcal A$, with $\mathcal A' \subseteq \mathcal
A$, if for each $b\in \mathcal A$ there is  $S_b = \{ a_i : 1 \le i
\le t\} \subseteq \mathcal A'$ such that $b \succeq \sum _i a_i,$
then $\mathcal A'$ generates $\mathcal A$ in the usual sense.
\end{lem}
\begin{proof}
For $b \in \mathcal A$, write $b = \sum a_i$, where $a_i \in \tT$, and $\sum _j b_{i,j} \preceq a_i$
for $b_{i,j} \in \mathcal A'$, implying $\sum _j b_{i,j} = a_i$ by Definition \ref{precedeq07} for $\tT$-systems, and
thus $b = \sum _i \sum _j b_{i,j}.$
\end{proof}

  \begin{defn}\label{mor}
Let $\mathcal A$ be a system. A \textbf{$\preceq$-morphism}  of $\mathcal A$-module
    pseudo-systems $$\varphi:
(\mathcal M, \tT_{\mathcal M}, (-), \preceq)\to (\mathcal M',
\tT'_{\mathcal M'}, (-)', \preceq')$$ is a map $\varphi: \mathcal M
\to \mathcal M'$ satisfying the following
properties  for $a_i \in \tT$ and   $b\preceq b' $  in $\mathcal M$,
 $b_i$  in $\mathcal
M$:
\begin{enumerate}\eroman
\item
$\varphi (\zero ) = \zero  .$
\item $ \varphi((-)b_1)=   (-)
\varphi(b_1);$
\item $ \varphi(b_1 + b_2) \preceq ' \varphi(b_1) + \varphi( b_2) ;$
\item  $ \varphi(a_1 b)=  a_1  \varphi( b) $.
\item $ \varphi(b) \preceq ' \varphi(b').$
\end{enumerate}
 By a \textbf{homomorphism} we mean the usual universal
algebra definition, i.e.,  in (iii), equality holds  instead
of~$\preceq$.
   \end{defn}

In many cases (such as for hypersystems) we also want to include the
condition $\varphi(\tT) \subseteq \tT'$, but there are instances for
which we do not want this condition to hold (for example the zero
morphism $a \mapsto \zero$, the morphism $a \mapsto a^\circ$, or
more generally, null morphisms of Definition \ref{modker} below).

\begin{rem}\label{morprop} $ $  \begin{enumerate}\eroman\item   $\varphi ( \mathcal M _{\operatorname{Null}} )\subseteq  \mathcal
 {M'} _{\operatorname{Null}}$
follows from conditions (i) and (v), since $\zero \preceq b$ implies
$\zero = \varphi(\zero) \preceq' \varphi( b)$.

\item  To show condition (ii), it is enough to have $\varphi((-)a_1)\preceq'  (-)
\varphi(a_1),$ by \cite[Proposition 2.41]{JuR1}.

\item   $  \varphi(b  (-)
c)\succeq \varphi(b),$ for all $c \succeq \zero,$ in view of (v) and
Lemma~\ref{circget7}.

\item  There also is a subtle issue concerning $\preceq$-morphisms  of
systemic modules; we would want $ \varphi(\tT) \cap \tT_{\mathcal
M'}$ to generate $(\mathcal M',+),$ in order for $ \varphi(\mathcal
M)$ to be a systemic module. \end{enumerate}
\end{rem}

\begin{lem} When $\preceq '$ is a PO (partial order) and  $\tT$
is a group,  Definition \ref{mor}(iv)  is implied by the formally
weaker condition
$$ \varphi(a b) \preceq ' a  \varphi( b), \quad \forall a \in \tT.
$$
\end{lem}
\begin{proof} $a \varphi( b) = a\varphi( a ^{-1} a b) \preceq '  a a^{-1} \varphi(a b) \preceq '  a  \varphi( b)
$, so equality holds at each stage.
\end{proof}

%

    Analogously, by a $\succeq$-\textbf{morphism} we use the same
definition as $\preceq$-morphism, except with (iii) now reading $$
\varphi(b_1 + b_2) \succeq ' \varphi(b_1) + \varphi( b_2) .$$

\begin{example}\label{morphmean}
Let us describe these notions for Example~\ref{precmain}; in the
process we see why we want to consider $\preceq$-morphisms rather
than just homomorphisms.
\begin{enumerate}\eroman\item  In supertropical
mathematics, a $\preceq$-morphism $f$ satisfies
\begin{equation}\label{mor1}
f(b_1+b_2)+\text{ghost} =
f(b_1)+f(b_2);
\end{equation}
\eqref{mor1} implies that either  $f(b_1+b_2)  = f(b_1)+f(b_2)$, or
$f(b_1)+f(b_2)$ is ghost, in which case either $f(b_1)= f(b_2)$, or
$f(b_1)$ is ghost of value greater than or equal to $f(b_2)$ (or
visa versa). In particular, this is the case for tropicalization of
the Puiseux series via the Puiseux valuation, and is one of our main
motivations in introducing $\preceq$-morphisms.
%

\item
For hypersystems, a  $\preceq$-morphism $f$ satisfies
\begin{equation}\label{morphmean3}
f(b_1\boxplus b_2)\subseteq f(b_1)\boxplus
f(b_2),
\end{equation}
the definition used in \cite[Definition~2.1]{CC} and
\cite[Definition~2.4]{GJL}.  This is intuitive when $f$ maps the
hyperring $\tT$ into itself. On the other hand, hyperring
$\succeq$-morphisms  which are not homomorphisms  seem to be
artificial; for an example, one could extend the identity on the
phase hyperfield to a map that doubles all non-singleton arcs around
the center.

Given  a hypersystem $({\mathcal A}= \mathcal P(\tT), \tT,
(-),\subseteq)$ and a  hypergroup morphism $f$ over $\tT$, it is
natural to extend
 $f$ to ${\mathcal A}$ via $$f( \{ a_i: a_i \in \tT\}) = \{ f(a_i): a_i \in
 \tT\}.$$
In this  case, if $f(b) (-) f(b') \succeq \zero,$ there is some
hypergroup element $a \in f(b) \cap f(b').$


\item
For fuzzy rings, in \cite[\S~1]{Dr}, also see \cite[Definition~2.17
]{GJL}, a homomorphism $$f: (K;+;\times, \vep _K,K_0)\to
(L;+;\times; \vep _L;L_0)$$ of fuzzy rings is defined as satisfying:
For any $\{a_1 ,\dots, a_n \} \in  K ^\times$  if $\sum _{i=1}^n a_i
\in K_0$ then $\sum _{i=1}^n f (a_i) \in L_0.$ Any
$\preceq$-morphism in our setting is a fuzzy homomorphism since
$L_0$ is an ideal, and thus $\sum _{i=1}^n f (a_i) \in f (\sum
_{i=1}^n a_i) + L_0 = L_0.$ The other direction might not hold. The
same reasoning holds for tracts of \cite{BB1}.
%
\item Another interesting example comes from valuation theory.
In \cite[Definition~8.8(ii)]{Row16}, valuations are displayed as
$\preceq$-morphisms of semirings, writing the target of the
valuation as a semiring (using multiplicative notation instead of
additive notation) via Green's relation of
Example~\ref{precmain}(viii).   Here $\varphi(b_1 b_2) =
\varphi(b_1) \varphi(b_2) .$ If we instead wrote $\varphi(b_1 b_2)
\preceq \varphi(b_1) \varphi(b_2) ,$ we would have a
quasi-valuation.
\end{enumerate}\end{example}

In conjunction with the hyperring theory, we are most interested in
$\preceq$-morphisms and homomorphisms, but at times 
we need the
restriction to homomorphisms. Occasional results can be formulated
for $\succeq$-morphisms, such as in Lemma \ref{lem: 3.14}(iii)
below.

\subsubsection{Direct sums and direct limits}$ $

The direct sum of $\tT$-modules, defined in the usual way, is
extended to pseudo-triples, \cite[\S 2.5.1]{JuR1}.

  \begin{defn}\label{cop}
    \begin{enumerate}\eroman
        \item
        The \textbf{direct sum} $\oplus _{i \in I}(\mathcal A _i , \tT_{ \mathcal A _i}, (-))$ of
         a family of pseudo-triples  over an index set $I$ (not necessarily finite) is defined as  $(\oplus \mathcal A _i, \tT_{\oplus \mathcal A _i},
        (-))$, where $\tT_{\oplus \mathcal A _i} = \cup   \tT_{ \mathcal A _i}$, viewed in $\oplus    \mathcal A _i$.
        \item
  The  \textbf{free} $\mathcal A$-module pseudo-triple $(\mathcal
  A^{(I)}, \tT_{\mathcal
  A^{(I)}}, (-))$ over a pseudo-triple $(\mathcal
  A,\tT,(-))$ is the direct sum of copies of $(\mathcal
  A,\tT,(-)).$
 \end{enumerate}
  If $(\mathcal A, \tT, (-),
  \preceq))$ is a system, we can extend  $\preceq$ componentwise to
  $\mathcal A^{(I)}$ to obtain the \textbf{free} $\mathcal A$-module
  system.
   \end{defn}

\begin{rem} When   $\preceq$ is   a PO on $ \mathcal A $, $\preceq$ is also  a PO on $ \mathcal A^{(I)}$,
seen componentwise. \end{rem}
%
%
%

\section{Systemic versions of basic module properties}

We want to find the systemic generalization of classical concepts of
module theory. As we shall see, this depends on which version we
use, i.e., the    switch negation map in the symmetrization given in
\S \ref{symm}, or taking a given surpassing negation map $(-)$ and
surpassing relation $\preceq$. These two different approaches give
rise to different theories.

\subsubsection{Notation}$ $

Let us fix some notation for the remainder of this paper. In what
follows, we let $\mathcal{A}= (\mathcal A, \tT, (-),
  \preceq)$ be a semiring system, and $\mathcal
M$ and $ \mathcal N$ always denote $\mathcal A$-systemic modules.
We write $\preceq$ generically for the appropriate surpassing PO in
a system. 

\subsection{Module theoretic notions}

 \begin{defn}\label{modker}
Let $\mathcal M$ and $\mathcal N$ be $\mathcal A$-systemic modules,
and $f: \mathcal M \to \mathcal N$ a $\preceq$-morphism.
\begin{enumerate} \eroman
    \item A submodule $\mathcal M'$ of $\mathcal M$ is   \textbf{$f$-null} if
$f(a) \in \mathcal N_{\operatorname{Null}}$ for all $a \in \mathcal
M'.$ The \textbf{null-module kernel} $\ker_{\Mod,\mathcal M} f$  of
$f$ is the sum of all $f$-null submodules of $\mathcal M$.
    \item
        A $\preceq$-morphism $f: \mathcal M\to \mathcal N$ is
    \textbf{null} if $f(\mathcal M)\subseteq \mathcal N_{\operatorname{Null}},$
    i.e., $\ker_{\Mod,\mathcal M} f = \mathcal M.$
\item
A $\preceq$-morphism $f$ is \textbf{null-monic}
(resp.~\textbf{null-epic}) when it satisfies the property that if
$fh$ is null (resp.~$hf$ is null) for a homomorphism $h$, then $h$
is null.

\item A $\preceq$-morphism $f$ is \textbf{N-monic}
  when it satisfies the property that
if $f(b) = f(b')$ for $b,b' \in \mathcal M$ then $b=b'$.
\end{enumerate}
 \end{defn}


\begin{rem}\begin{enumerate}\eroman
\item  $\mathcal{M}_{\operatorname{Null}}\subseteq \ker_{\Mod,\mathcal M} f
$, by Remark~\ref{morprop}.

 \item
    Being the sum of submodules of $\mathcal M$,  $\ker_{\Mod,\mathcal M} f  $ is a submodule
    of $\mathcal M$, which is $f$-null when $f$ is a homomorphism, but  need
    not be $f$-null when $f$ is just a $\preceq$-morphism. One could
    have $f(a_1)+ f(a_2) \succeq \zero$ whereas $f(a_1+a_2) \not \succeq
    \zero$.
\end{enumerate}
\end{rem}

 \begin{lem}\label{kertriv}    A $\preceq$-morphism   $f:  \mathcal{M} \to \mathcal N$ is null-monic if and only if the null-module kernel  of $f$
 is a subset of $\mathcal{M}_{\operatorname{Null}}.$
\end{lem}
 \begin{proof} $(\Rightarrow)$ For any $f$-null submodule $\mathcal M'$ of $\mathcal
 M$,
 consider the identity map $ h: \mathcal M' \to
\mathcal M'.$ Then $fh$ is null, implying $h$ is null. In
particular, $\mathcal M'= \mathcal M'_{\operatorname{Null}}\subseteq
\mathcal{M}_{\operatorname{Null}},$ and hence the null-module kernel
of $f$ is a subset of $\mathcal{M}_{\operatorname{Null}}.$

$(\Leftarrow)$ Suppose $fh$ is null, for a homomorphism $h:
\mathcal{K} \to \mathcal M.$ Then   $f(h(\mathcal K)) \subseteq
\mathcal{N}_{\operatorname{Null}}$. This implies that
$$h(\mathcal{K}) \subseteq \ker_{\Mod,\mathcal M} f \subseteq \mathcal
  M_{\operatorname{Null}},$$
 proving that $f: \mathcal M \to \mathcal N$ is null-monic.
  \end{proof}

Next, we define some notation which we will use later in defining
projective modules.

\begin{defn}\label{precont1}  Let $f:\mathcal M \to \mathcal N$ be a $\preceq$-morphism of $\mathcal A$-systemic modules $\mathcal M$ and $\mathcal N$. We define the following two sets:
\[
f( {\mathcal M} )_\preceq = \{ b \in \mathcal N: b \preceq f(a)
 \textrm{ for some }a \in {\mathcal M} \},\quad f( {\mathcal M} )_\succeq = \{ b \in \mathcal
N: b \succeq f(a) \textrm{ for some }a \in {\mathcal M}\}.
\]
\begin{enumerate}\eroman
\item
$f:  \mathcal M \to \mathcal N$ is $\preceq $-\textbf{onto} if $f(
{\mathcal M} )_\preceq \, =\mathcal N,$  i.e., for   every $b' \in
{\mathcal N}$ there exists $b \in  {\mathcal M}$, for which $b'
\preceq f(b) $.
\item
$f:  \mathcal M \to \mathcal N$ is h-\textbf{onto} if $f$ is a
$\preceq $-onto homomorphism.
\item
$f:  \mathcal M \to \mathcal N$ is $\succeq $-\textbf{onto} if $f(
{\mathcal M} )_\succeq \,  =\mathcal N,$ i.e., if for   every $b'
\in {\mathcal N}$ there is $b \in  {\mathcal M}$ such that $b'
\succeq f(b )$.
\end{enumerate}
\end{defn}

Although $\preceq$-onto and h-onto seem more appropriate for this
paper, giving stronger results for projectivity, A.~Connes and
C.~Consani \cite{CC2} use a definition more in line with $\succeq
$-onto and which seem more appropriate for homology, taking modules
over the Boolean semiring $\mathbb B$, whose symmetrization has some
properties parallel to the supertropical semialgebra. This
connection is to be discussed in detail in \cite{JMR1}.

 Most of our results  hold already for $\preceq $-onto
$\preceq $-morphisms and homomorphisms.

\begin{example}\label{morphm} In the supertropical setting, $f:  \mathcal M \to \mathcal N$ is $\preceq
    $-onto iff for every element $b$ of $\mathcal N$ there is $c\in \mathcal
    M$ such that $b+ \text{ghost}= f(c),$ which often is easy to
    satisfy when $c$ is a large enough ghost. $\succeq
    $-onto says that $b = f(c)+ \text{ghost}$, which for $b$ tangible says
    $b= f(c)$.


    For fuzzy rings, the condition says something about how $f(K_0)$ sits inside    $L_0$, notation as in \cite{Dr}
    .
\end{example}


\begin{lem}\label{monep10}
Let $f:\mathcal M \to \mathcal N$ be a $\preceq$-morphism of
$\mathcal A$-systemic modules $\mathcal M$ and $\mathcal N$. Then
\begin{enumerate} \eroman
\item
$f( {\mathcal M} )_\succeq $ is a submodule of $\mathcal N$. Moreover, $f$ is $\succeq $-onto, if for every $b \in \tT_{\mathcal N}$ there is $a \in {\mathcal M}$ such that $f(a )\preceq b$.

\item   $f( {\mathcal M} )_\preceq $ is a submodule of $\mathcal N$
for any  homomorphism $f:  \mathcal M \to \mathcal N$.

\end{enumerate}
 \end{lem}
\begin{proof} (i) $f(\mathcal M)_\succeq$ is clearly closed under the action of $\mathcal A$ and contains $\zero$. If $b_i \in f( {\mathcal M} )_\succeq $ for $i = 1,2$ then
writing $b_i \succeq f(a_i)$, we have $$b_1 + b_2 \succeq f(a_1)+f(a_2) \succeq f(a_1+a_2).$$ This shows that $f(\mathcal M)_\succeq$ is also closed under addition. The second assertion follows from the fact that $\tT_{\mathcal N}$ generates $\mathcal N.$ In fact, for any $b \in \mathcal N$, there exist $b_i \in \tT_\mathcal N$ such that $b=\sum_i b_i$. But, from the given condition, we can find $a_i \in \mathcal M$ such that $f(a_i) \preceq b_i$ and hence we have
\[
b=\sum_i b_i \succeq \sum f(a_i) \succeq f\left(\sum_i a_i\right).
\]

(ii)  One can easily check that $f(\mathcal M)_\preceq$ is closed
under the action of $\mathcal A$ and contains $\zero$. Suppose that
$b_1, b_2 \in f( {\mathcal M} )_\preceq$, i.e., there exist $a_1,a_2
\in \mathcal M$ such that $b_i \preceq f(a_i)$ for $i=1,2$. Since
$f$ is a homomorphism, it follows that
\[
b_1+b_2 \preceq f(a_1)+f(a_2)=f(a_1+a_2).
\]
This shows that $b_1+b_2 \in f(\mathcal M)_\preceq$ and hence $f(\mathcal M)_\preceq$ is also closed under addition.



 \end{proof}


\begin{defn}\label{precspli1}
\begin{enumerate}\eroman
\item
An onto homomorphism $\pi:\mathcal M \to \mathcal N$ is an
\textbf{N-quasi-isomorphism} if   $\pi$ is also N-monic.
\item
A  $\preceq$-onto $\preceq$-morphism $\pi:\mathcal M \to \mathcal N$
is a \textbf{$\preceq$-quasi-isomorphism} if $\pi$ is also
null-monic.
 \end{enumerate}
\end{defn}

\subsection{Congruences}$ $

Recall that a congruence on $\mathcal M$ is an equivalence relation
which preserves all of the operators; i.e., it is a subsystem of
$\mathcal M \times \mathcal M$ that contains the diagonal $\diag
_{\mathcal M}  := \{ (a,a) : a \in \mathcal M\}$ and is  symmetric
and transitive.

\subsection{$\preceq$-split  and h-split epics}\label{splitont}$ $

We recall a standard definition.

\begin{defn}\label{split8}
Let $\pi: \mathcal M\to \mathcal N$ be an onto homomorphism. We say
that $\pi: \mathcal M\to \mathcal N$ \textbf{splits}
 if there is a homomorphism $\nu:\mathcal N
\to \mathcal  M$ such that $ \pi \nu =\one_{\mathcal N} $.
 \end{defn}

In classical algebra, $\nu$ must be monic, and   any split epic
gives rise to an exact sequence.

\begin{example} If $\mathcal M= \mathcal N \oplus \mathcal  N',$ then the canonical projection $\mathcal M \to \mathcal  N$    splits
via the natural injection $\nu: \mathcal N\to \mathcal M$.
\end{example}

 This is trickier in the theory of
systems since, as we shall see, the analog of splitting need not
involve direct sums; a similar issue has been already observed in
tropical algebra, cf.~\cite[\S 2]{MJ}. Accordingly, we want to
weaken the definition, and consider its implications.

We write $f\preceq g$ for $\preceq$-morphisms $f,g:\mathcal M \to
\mathcal N$,
 if $f(b) \preceq g(b)$ for all $b \in \mathcal M$. Now, we weaken
Definition \ref{split8} as follows:

\begin{defn}\label{precspli} $ $
\begin{enumerate}\eroman
    \item
We say that a $\preceq$-morphism $\pi: \mathcal M\to \mathcal N$
\textbf{$\preceq$-split
s} (resp.~\textbf{h-splits}) if there is a
$\preceq$-morphism  (resp.~\textbf{homomorphism}) $\nu:\mathcal N
\to \mathcal M$ such that $\one_{\mathcal  N} \preceq \pi \nu $. In
this case, we also say that $ \nu $ \textbf{$\preceq$-splits}
(resp.~\textbf{$\preceq$ h-splits})~$\pi$, and that $\mathcal N$ is
a $\preceq$-\textbf{retract}  (resp.~\textbf{h-retract}) of
$\mathcal M$. \textbf{$\succeq$-splits} (resp.~\textbf{$\succeq$
h-splits}) is defined analogously, with  $\one_{\mathcal  N} \succeq
\pi \nu $.
\item
Let $f$ be any of $\{ \preceq$-morphism, homomorphism
$\succeq$-morphism, homomorphism$\}$.
\begin{enumerate}
    \item
$f: \mathcal M\to \mathcal M$ is
$\preceq$-\textbf{idempotent} if $f^2 \succeq f.$
\item
$f: \mathcal M\to
\mathcal M$ is $\tT$-\textbf{idempotent} if $f^2(a)=f(a),$ for all
$a \in \tT.$
\item
$f: \mathcal M\to \mathcal M$ is $(\tT,\preceq)$-\textbf{idempotent}
if $f$ is both $\tT$-idempotent and $\preceq$-idempotent.
\end{enumerate}
\end{enumerate}

   \end{defn}

 \begin{lem} \label{lem: 3.14}
 \begin{enumerate} \eroman
\item If  $\pi:\mathcal M \to \mathcal N$
and $\nu:\mathcal N \to \mathcal M$ are $\preceq$-morphisms with
$\one_{\mathcal  N} \preceq \pi \nu,$ then $\pi$ is  $\preceq$-onto,
and $\nu \pi $ is $\preceq$-idempotent.
\item If $\pi:\mathcal M \to \mathcal N$ is a homomorphism and $\nu:\mathcal N \to \mathcal M$ is a $\preceq$-morphism with $\one_{\mathcal  N} \preceq \pi \nu,$
 then $\one_{\mathcal M} (-)\nu \pi $  is $\preceq$-idempotent.
\item If  $\pi:\mathcal M \to \mathcal N$ and $\nu:\mathcal N \to \mathcal M$ are $\succeq$-morphisms with $\one_{\mathcal  N} \succeq \pi \nu
 $, then $\nu$ is null-monic and   $\one (-)\nu \pi $  is $\preceq$-idempotent.
 \item If  $\pi:\mathcal M \to \mathcal N$ is a homomorphism and $\nu:\mathcal N \to \mathcal M$ is a $\preceq$-morphism with $ a  = \pi \nu (a)$ for all
$a \in \tT_{\mathcal  N},$ then $\pi$ is onto, and $\nu
\pi $ is $\tT$-idempotent.

%
\end{enumerate}
 \end{lem}
 \begin{proof}
 (i) For any $b \in \mathcal N$, we have that $ b \preceq \pi (\nu (b)).
 $ This shows that $\mathcal N =\pi(\mathcal M)_\preceq$ and hence $\pi$ is $\preceq$-onto.
 Furthermore, $\nu \pi \nu \pi = \nu (\pi \nu )\pi \succeq \nu \one_{\mathcal  N}\pi  = \nu
 \pi$.

 (ii) For any $b \in \mathcal M$, let $c = \pi (b),$ so $(\one_{\mathcal M}  (-) \nu \pi ) (b)
 = b(-)  \nu   (c) $. Also $ \pi  \nu(c) \succeq c$ implies $ \pi  \nu(c) (-) c \succeq
 \zero$ by Lemma~\ref{circget},
 so $c(-) \pi  \nu(c)  \succeq \zero$, and thus $\nu(c(-) \pi  \nu(c)  )\succeq \zero
 $.
  Hence
\begin{equation}\label{M_21}
\begin{split}
(\one_{\mathcal M}(-)\nu \pi)(\one_{\mathcal M} (-) \nu \pi)(b) & =  (\one_{\mathcal M} (-) \nu \pi)( b(-)  \nu
(c)) \\
 & =  b(-)  \nu   (c) (-)\nu(c (-) \pi  \nu(c))  \succeq b(-)  \nu
(c)= (\one_{\mathcal M} (-) \nu \pi)(b).
\end{split}
\end{equation}

(iii)  For the first assertion, from Lemma \ref{kertriv}, it is
enough to show that $\ker_{\Mod,\mathcal N} \nu \subseteq \mathcal
N_{\operatorname{Null}}$. If $b \in \ker_{\Mod,\mathcal N} \nu$,
then $\nu(b) \succeq \zero$. Since $\one_{\mathcal  N} \succeq \pi
\nu
 $, we further have that
 \[
 b \succeq \pi \nu (b) \succeq \pi(\zero) \succeq \zero
 \]
 and hence $b \in \mathcal N_{\operatorname{Null}}$, showing that $\ker_{\Mod,\mathcal N} \nu \subseteq \mathcal
 N_{\operatorname{Null}}$. Also
  $$
  (\one_{\mathcal M}(-)\nu \pi)(\one_{\mathcal M} (-) \nu \pi) \ \succeq \ \one_{\mathcal M}\, +\, \nu \pi \nu \pi\, (-)\, \nu \pi \, (-)\, \nu
 \pi \succeq  \one_{\mathcal M}\,+ \,  \nu \pi \,(-)\,\nu \pi \,(-)\,\nu
 \pi \succeq  \one_{\mathcal M} (-)\nu
 \pi,$$
 showing that $(\one_{\mathcal M}(-)\nu\pi)$ is $\preceq$-idempotent.

 (iv) For any $b \in \mathcal N$, we can write $b =\sum_i a_i$ for   $a_i \in \tT_{\mathcal  N}$. Since, we assume that
 $a=\pi\nu(a)$ for any $a \in \tT_{\mathcal  N}$ and $\pi$ is a homomorphism, we have that
 \[
 b=\sum_i a_i = \sum_i \pi\nu(a_i)=\pi(\sum_i \nu(a_i)),
 \]
showing that $\pi$ is onto. Furthermore, for any $b \in \tT_{\mathcal M}$, we have that
\[
\nu\pi\nu\pi(b)=\nu(\pi\nu)(\pi(b))=\nu\pi(b),
\]
showing that $\nu\pi$ is $\tT$-idempotent.

 \end{proof}
%


 \begin{defn}\label{definition: direct sum} A systemic module  $\mathcal M =
(\mathcal M, \tT_{\mathcal M}, (-),\preceq)$ is the (finite)
\textbf{$\preceq$-direct sum} of systemic modules $(\mathcal M_i,
\tT_{\mathcal M_i}, (-),\preceq),\quad i \in I$ ($I$ finite), if
there are $\preceq$-morphisms $\pi _i : \mathcal M \to \mathcal M_i$
as well as $\preceq$-morphisms $\nu_i: \mathcal M_i \to \mathcal M$
that $\preceq$-split  $\pi _i$, for which $\one _ \mathcal M \preceq
\sum \nu_i\pi _i$, $\one _ {\mathcal M_i} \preceq \pi _i \nu_i$, and
$\zero _ \mathcal M \preceq
 \pi _j \nu_i$ for all $i \ne j.$

The analogous definition, \textbf{h-direct sum}, is for homomorphisms instead of
 $\preceq$-morphisms as follows:

$(\mathcal M, \tT_{\mathcal M}, (-),\preceq)$ is the (finite)
\textbf{h-direct sum} of  $(\mathcal M_i, \tT_{\mathcal M_i},
(-),\preceq),\quad i \in I$ ($I$ finite), if there are homomorphisms
$\pi _i : \mathcal M \to \mathcal M_i$ as well as homomorphisms
$\nu_i: \mathcal M_i \to \mathcal M$ that $\preceq$-split  $\pi _i$,
for which $\one _ \mathcal M \preceq \sum \nu_i\pi _i$, $\one _
{\mathcal M_i} \preceq \pi _i \nu_i$, and $\zero _ \mathcal M
\preceq
 \pi _j \nu_i$ for all $i \ne j.$

$\mathcal M = (\mathcal M, \tT_{\mathcal M}, (-),\preceq)$ is the
(finite) \textbf{$\succeq$-direct sum} of $(\mathcal M_i,
\tT_{\mathcal M_i}, (-),\preceq),\quad i \in I$ ($I$ finite), if
there are $\succeq$-morphisms $\pi _i : \mathcal M \to \mathcal M_i$
as well as $\succeq$-morphisms $\nu_i: \mathcal M_i \to \mathcal M$
that $\succeq$-split  $\pi _i$, for which $\one _ \mathcal M \succeq
\sum \nu_i\pi _i$, $\one _ {\mathcal M_i} \succeq \pi _i \nu_i$, and
$\zero _ \mathcal M \succeq
 \pi _j \nu_i$ for all $i \ne j.$
\end{defn}

Then we have the following.

\begin{thm}\label{splitdir}
Let $\pi: \mathcal M \to \mathcal N$ be a homomorphism. If $ \nu $
$\preceq$-splits $\pi$ (resp.~h-splits $\pi$), then:
\begin{enumerate}\eroman
\item
 $\mathcal M$ is
the $\preceq$-direct sum (resp.~h-direct sum) of $\mathcal M_1: =
\pi(\mathcal M)$ and $\mathcal M_2: = (\one_\mathcal M (-)\nu
\pi)(\mathcal M)$ with respect to the $\preceq$-morphisms
(resp.~homomorphisms) $\pi_1 =\pi, \, \nu _1  = \nu ,\,$ $\pi_2 =
(\one _\mathcal M (-)\nu \pi),$ $ \nu_2 = \one_{\mathcal M_2}$.
\item $\mathcal M$ is
the $\preceq$-direct sum (resp.~h-direct sum) of $\mathcal M_1= \nu
\pi   (\mathcal M)$ and $\mathcal M_2 = \ker_{\Mod,\mathcal M}\pi$,
with respect to $\nu _i = \one_{\mathcal M_i}$ for $ i = 1,2$.
\end{enumerate}
\end{thm}
 \begin{proof}
 For notational convenience, we write $\one_\mathcal M= \one_{\mathcal M_i} =\one$.
  $\one (-)\nu \pi $  is $\preceq$-idempotent, by Lemma~\ref{lem: 3.14}(ii).

(i) $\pi_1$ is $\preceq$-onto by Lemma~\ref{lem: 3.14}(i),  and
$\pi_2$ is $\preceq$-onto by the definition of $\mathcal M_2$.

Next, we show that $\nu_2$ $\preceq$-splits $\pi_2$.
Take $b_2 \in \mathcal M_2$. This means that there exists $b_1 \in
\mathcal M$ such that $b_2=b_1(-)\nu\pi(b_1)$, and now one observes
\begin{equation}\label{M_2}
\pi _2 \nu_2(b_2)=\pi_2(b_2)= (\one (-) \nu \pi)^2(b_1)\succeq (\one
(-) \nu\pi)(b_1)= b_2,
\end{equation}
since $\one (-)\nu\pi$ is $\preceq$-idempotent
by Lemma~\ref{lem: 3.14}(ii). 

We now show the remaining conditions. One can easily see the following:
\[
\nu _1 \pi_1 (b) +  \nu _2 \pi_2 (b)  =    \nu\pi (b) + (\one(-)\nu
\pi)(b)=b+(\nu\pi(b)(-)\nu\pi(b)) \succeq b, \quad b  \in \mathcal
M,
\]
showing that $\one_\mathcal M \preceq \nu_1\pi_1+\nu_2\pi_2$.

Finally, we have for $b = (\one (-)\nu \pi)b' \in \mathcal M_2 ,$

\[\pi_1 \nu_2(b)=
 \pi (\one(-)\nu\pi)(b') = (\pi  (-)\pi\nu\pi)(b') \succeq \pi(b')(-)\pi(b')\succeq \zero,\]

and similarly, for $b \in \mathcal M_1$, 
\[\pi_2 \nu_1 (b) = (\one (-)\nu \pi)(\nu(b)) = \nu(b) (-)\nu\pi\nu(b) \succeq \nu(b) (-)\nu(b) \succeq 0.
\]

%
The proof for h-splitting is analogous since $\pi_i$ and $\nu_i$ are homomorphisms.

(ii) The same sort of verifications as in (i), but easier. Now we
take $\pi_1 = \nu\pi,$ $\pi_2 = \one(-)\nu\pi,$ and $\nu_i$ to be
the canonical injection for $i=1,2,$ which we write as the identity
map. Then $\pi_1(\mathcal M) = \mathcal M_1$ and, for every $b \in
\mathcal M,$  $\pi_2(b) = b (-) \nu\pi(b)\in \mathcal M_2$ since $
\pi\nu \pi \succeq \pi$. But if $b_2 \in \mathcal M_2$ then
$$ (\one(-)\nu\pi)(b_2) = b_2 (-)\nu\pi(b_2) \succeq b_2,$$ so $\pi_2$ is
$\preceq$-onto.

 $\nu_2$ $\preceq$-splits $\pi_2$, since
 $$\pi_2 \nu_2(b_2)=
(\one(-)\nu\pi)(b_2) = b_2 (-) \nu\pi (b_2) \succeq  b_2  (-) \zero
\succeq b_2.$$
Furthermore, one can easily see that
\[
\nu _1 \pi_1 (b) +  \nu _2 \pi_2 (b)  =    \nu\pi (b) + (\one(-)\nu
\pi)(b)=b+(\nu\pi(b)(-)\nu\pi(b)) \succeq b, \quad b  \in \mathcal
M.
\]

Finally, for $b = (\one (-)\nu \pi)b' \in \mathcal M_2 ,$

for $b = \nu \pi b' \in \mathcal M_1 ,$ $$\pi_2
\nu_1(b)=(\one(-)\nu\pi)(b) = (\one(-)\nu\pi)\nu \pi(b') \succeq \nu
\pi(b') (-) \nu \pi(b') \succeq \zero,$$

and for $b  \in \mathcal M_2 ,$ $\pi_1 \nu_2(b)=
 \pi  (b )  \succeq \zero.$
\end{proof}
%

\section{$\preceq$-projective and $\succeq$-projective modules}\label{proj}$ $

 We are ready to define several versions of $\preceq$-projective  systemic modules, as well
 as
  $\succeq$-projective  modules over ground $\tT$-systems. This
  encompasses results of \cite{KN}, in view of Remark~\ref{sysgene}. 

\begin{defn}\label{precproj} 
 (See \cite{DeP,Ka3,KN,Tak} for comparison)
\begin{enumerate}\eroman
\item
A systemic module  $\mathcal P: =   (\mathcal P,\tT_{\mathcal
P},(-),\preceq)$ is \textbf{projective} if for any  onto
homomorphism of  systemic modules $h: \mathcal M \to \mathcal M'$,
every homomorphism $ f: \mathcal P \to \mathcal M'$ \textbf{lifts}
to a homomorphism $\tilde f: \mathcal P \to\mathcal M$, in the sense
that $h\tilde f =  f.$
\item
$\mathcal P$ is  $\preceq$-\textbf{projective} if for any
$\preceq$-onto $\preceq$-morphism  $h: \mathcal M \to \mathcal M',$
every  $\preceq$-morphism $ f: \mathcal P \to \mathcal M'$
$\preceq$-\textbf{lifts} to a $\preceq$-morphism $\tilde f: \mathcal
P \to \mathcal M$, in the sense that $f \preceq h\tilde f  .$
\item
$\mathcal P$ is  $(\preceq,h)$-\textbf{projective} if for any
$\preceq$-onto homomorphism  $h: \mathcal M \to \mathcal M',$ every
$\preceq$-morphism $ f: \mathcal P \to \mathcal M'$
$\preceq$-\textbf{lifts} to a $\preceq$-morphism $\tilde f: \mathcal
P \to \mathcal M$, in the sense that $f \preceq h\tilde f  .$
\item
$\mathcal P$ is   \textbf{h-projective} if for any  $\preceq$-onto
homomorphism  $h: \mathcal M \to \mathcal M',$ every homomorphism $
f: \mathcal P \to \mathcal M'$ $\preceq$-\textbf{lifts} to a
homomorphism $\tilde f: \mathcal P \to \mathcal M$, in the sense
that $f \preceq h\tilde f  .$
\item
 $\mathcal P$ is  $\succeq$-\textbf{projective} is defined
analogously to $\preceq$-projective, with $\succeq$ replacing
$\preceq$ where appropriate. In other words, for any $\succeq$-onto
$\preceq$-morphism  $h: \mathcal M \to \mathcal M',$ every
$\succeq$-morphism $ f: \mathcal P \to \mathcal M'$
$\succeq$-\textbf{lifts} to a $\succeq$-morphism $\tilde f: \mathcal
P \to \mathcal M$, in the sense that $f \succeq h\tilde f  .$
\end{enumerate}

%

\end{defn}

Note that the subtleties in these versions: $ \preceq $-projective
implies $(\preceq,h)$-projective, but the condition could fail in
the important case of free systems, cf. Remark~\ref{ccom}(ii) below.
The
  definition for h-projective provides the most  results, but these are less inclusive since $\preceq$-projective systems
are more general, and still satisfy many results analogous to the
projective theory.

\subsection{Basic properties of $\preceq$-projective and h-projective
systems}\label{prsys}$ $

 \begin{lem} \label{lemma: free system is projective}
 The free $\mathcal A$-systemic module  $\mathcal F := \mathcal A^{(I)} $ is   projective, $(\preceq,h)$-projective, h-projective, and $\succeq$-projective.
 \end{lem}
  \begin{proof} We take the usual
 argument of lifting a set-theoretical map from the base $\{e_i: i \in I
 \}$ of $\mathcal F,$ in these three respective contexts.
Namely, choosing $x_i \in \mathcal M$ for which $h(x_i) =f(e_i)$
(resp.~$h(x_i) \succeq f(e_i)$, $h(x_i) \preceq f(e_i)$) and
defining a homomorphism $\tilde f: \mathcal F \to \mathcal M$ by
$\tilde f(e_i) = x_i,$ we have the three respective comparisons:

$$ f\left(\sum a_i e_i\right) = \sum f(a_i e_i)=
\sum a_i f(e_i)= \sum a_i h(x_i) = h\left(\sum a_i x_i\right)=
h\left(\sum a_i \tilde f(e_i)\right) = h \tilde f\left(\sum a_i
e_i\right) ,$$ proving   $f = h \tilde
 f.$

$$ f\left(\sum a_i e_i\right) \preceq \sum f(a_i e_i)=
\sum a_i f(e_i)\preceq \sum a_i h(x_i) = h\left(\sum a_i x_i\right)=
h\left(\sum a_i \tilde f(e_i)\right) = h \tilde f\left(\sum a_i
e_i\right) ,$$ proving   $f \preceq h \tilde
 f.$

$$ f\left(\sum a_i e_i\right) \succeq \sum f(a_i e_i)=
\sum a_i f(e_i)\succeq \sum a_i h(x_i) \succeq h\left(\sum a_i
x_i\right)= h\left(\sum a_i \tilde f(e_i)\right) = h \tilde
f\left(\sum a_i e_i\right),$$ proving   $f \succeq h \tilde
 f.$

 For h-projective see Remark~\ref{ccom}(i).

\end{proof}

\begin{rem}\label{ccom} $ $
\begin{enumerate} \eroman
\item In the proof,  we see, surprisingly, that any $\preceq$-morphism $ f:
\mathcal F \to \mathcal M$ can be $\preceq$-lifted to a homomorphism
$\tilde f: \mathcal F \to \mathcal M$ (since    $\tilde f$ is a
homomorphism by definition). This proves the remaining case of
h-projective in Lemma \ref{lemma: free system is projective}.

\item  The free $\mathcal A$-module  need not be
$\preceq$-projective, since the satisfaction of the first two
equations in the proof require $h$ to be a homomorphism!
\end{enumerate}\end{rem}

\subsubsection{Characterizations of $\preceq$-projective and
h-projective systemic modules}$ $

 Similar arguments as in \cite[\S 17]{golan92} show that
the following are equivalent for a systemic module  $\mathcal P$:
\begin{enumerate} \eroman
\item
$\mathcal P$  is  projective.
\item Every homomorphism onto   $\mathcal P$ splits.
\item There is an onto homomorphism from a free system to $\mathcal P$ that splits.
\item The functor
$\Hom (\mathcal P,\underline{\phantom{w}}) $ sends   onto
homomorphisms to  onto homomorphisms.
\end{enumerate}

 (iii)  is the condition used in \cite{KNZ} to define projective
modules. We extend this to $\preceq$. Define $\Mor_\preceq (\mathcal
M, \mathcal N)$ to be the  set of $\preceq$-morphisms from $\mathcal
M$ to $\mathcal N$, and its subset $\Hom (\mathcal M, \mathcal N)$
to be the  homomorphisms.

 \begin{prop}\label{epicspl}
The following are equivalent for a  systemic module  $\mathcal P$:
\begin{enumerate} \eroman
\item
$\mathcal P$  is  $(\preceq,h)$-projective.
\item
Every   $\preceq$-onto homomorphism to   $\mathcal P$
$\preceq$-splits.
\item
There is a $\preceq$-onto homomorphism from a free system to
$\mathcal P$ that  $\preceq$-splits.
\item
Given a $\preceq$-onto $\preceq$-morphism $h:\mathcal M \to \mathcal
M'$, the map $ \Mor_\preceq (\mathcal P,h): \Mor_\preceq(\mathcal P,
\mathcal M) \to \Mor_\preceq(\mathcal P,\mathcal M')$ given by $ g
\mapsto hg$ is $\preceq$-onto.
\end{enumerate}
 \end{prop}  \begin{proof}

  $((i)\Rightarrow(ii))$ Given a $\preceq$-onto homomorphism $h: \mathcal M \to \mathcal P$, the identity map $1_{\mathcal P}$ $\preceq$-lifts to
  a
  $\preceq$-morphism $g: \mathcal P \to \mathcal M$ satisfying  $
  1_{\mathcal P} \preceq h g  $.

  $((ii)\Rightarrow(iii))$ A fortiori, since we can define a $\preceq$-onto homomorphism from a free system to $\mathcal P$ by taking a base $\{e_i\}$ of a free system and sending the $e_i$ elementwise to the $\preceq$-generators of $\mathcal P$ as in the proof of Lemma \ref{lemma: free system is projective}.

 $((iii)\Rightarrow (i))$ Take a free  systemic module  $\mathcal
  F,$   with the projection   $\pi: \mathcal F \to \mathcal
  P$ which by hypothesis $\preceq$-splits, with $\nu : \mathcal P \to \mathcal
  F.$ Let $h: \mathcal M \to \mathcal M'$ be a $\preceq$-onto homomorphism. Then, for any $\preceq$-morphism $f: \mathcal P \to \mathcal M'$, we can $\preceq$-lift $ f
  \pi$ to $\tilde f : \mathcal F \to \mathcal M,$ i.e., $f \pi
  \preceq h \tilde f.$ Since $1_\mathcal P \preceq \pi\nu$, we have that
 \[
f  \preceq  f(\pi \nu) = (f \pi )\nu \preceq h(\tilde f \nu) ,
 \]
proving $\tilde f \nu$ $\preceq$-lifts $f.$

 $((i)\Leftrightarrow(iv))$ This directly follows from the definition. In fact, let $h:\mathcal M \to \mathcal M'$ be a $\preceq$-onto homomorphism. Then we have:
 \[
\Mor_\preceq (\mathcal P,h): \Mor_\preceq(\mathcal P, \mathcal M)
\to \Mor_\preceq(\mathcal P,\mathcal M'),\] given by $ g \mapsto hg.
$ For notational convenience, let $\varphi:=\Mor_\preceq (\mathcal
P,h)$, $A:=\Mor_\preceq(\mathcal P, \mathcal M)$, and
$B:=\Mor_\preceq(\mathcal P,\mathcal M')$. Then $\varphi$ is
$\preceq$-onto if and only if $\varphi(A)_\preceq = B$. Now, for any
$f \in B$, since $\mathcal P$ is $(\preceq,h)$-projective, there exists
$\tilde{f} \in A$, such that $f \preceq
h\tilde{f}=\varphi(\tilde{f})$. This shows that $\varphi$ is
$\preceq$-onto as desired.
 \end{proof}

\begin{remark}\label{remark: for later use}
One can easily see from the above proof on $((i) \implies (ii))$ that if $\mathcal{P}$ is $\preceq$-projective then every $\preceq$-onto morphism to $\mathcal{P}$ $\preceq$-splits.
\end{remark}
%
%
%
%
More appropriate in later research is the h-version:

 \begin{prop}\label{epicspl1}
The following are equivalent for a  systemic module  $\mathcal P$:
\begin{enumerate} \eroman
\item
$\mathcal P$  is  h-projective.
\item
Every $\preceq$-onto homomorphism to   $\mathcal P$ h-splits.
\item There is a $\preceq$-onto homomorphism from a free system to $\mathcal P$ that  h-splits.
\item Given a $\preceq$-onto homomorphism $h:\mathcal M \to \mathcal M'$, $ \Hom (\mathcal P,h): \Hom(\mathcal P, \mathcal M) \to \Hom(\mathcal P,\mathcal M')$ given by $ g \mapsto
hg$ is $\preceq$-onto.
\end{enumerate}
 \end{prop}
\begin{proof}
The proof   is analogous to that of
 Proposition~\ref{epicspl}. In fact, one may follow the proof of Proposition~\ref{epicspl} by replacing $\preceq$-splits with h-splits and also by using the fact that the composition $fg$ is a homomorphism if $f$ and $g$ are homomorphisms.
 \end{proof}

%
We also have the $\succeq$ version.
 \begin{prop}\label{epicspl10}
The following are equivalent for a systemic module  $\mathcal P$:
\begin{enumerate} \eroman
\item
$\mathcal P$  is   $\succeq$-projective.
\item Every $\succeq$-onto  $\succeq$-morphism to $\mathcal P$ $\succeq$-splits.
\item There is a $\succeq$-onto $\succeq$-morphism  from a free system to $\mathcal P$ that  h-$\succeq$-splits.
\item The functor $\Hom (\mathcal P,\underline{\phantom{w}})$
 sends  $\succeq$-onto $\succeq$-morphisms to
$\succeq$-onto $\succeq$-morphisms.
\end{enumerate}
 \end{prop}  \begin{proof} Analogous to the proof of Proposition~\ref{epicspl}, where we reverse $\preceq$ and $\succeq$ and apply Lemma \ref{lemma: free system is projective}
 (taking $\succeq$-morphisms instead of homomorphisms).  \end{proof}

 \begin{lem}  [as in {\cite[Proposition~17.19]{golan92}}] \label{lemma: project direct sum}
    A direct sum $\sum \mathcal P_i$ of  systemic modules  is
 \textbf{projective} (resp.~\textbf{$\preceq$-projective}, $\preceq$-\textbf{h-projective}, \textbf{h-projective}, \textbf{$\succeq$-projective}) if and only if  each $\mathcal P_i$ is projective
 (resp.~$\preceq$-projective, $\preceq$-h-projective, h-projective, $\succeq$-projective).
 \end{lem}
 \begin{proof} Formal, according to components.
\end{proof}

One can sharpen this assertion.
 \begin{prop}  \label{projspl} If $\pi : \mathcal Q \to \mathcal P$  is a $\preceq$-split (resp. h-split)
  $\preceq$-morphism (resp. homomorphism) and $\mathcal Q$ is $\preceq$-projective (resp.~h-projective), then $\mathcal P$ is also
$\preceq$-projective (resp.~h-projective).
 \end{prop}
 \begin{proof}
 We first prove the case when $\mathcal Q$ is $\preceq$-projective. We write a $\preceq$-splitting map $\nu: \mathcal P \to \mathcal Q$ as in Definition~\ref{precspli}.
 For any $\preceq$-onto morphism  $$h: \mathcal M \to \mathcal
 M',$$
and every $\preceq$-morphism $ f: \mathcal P \to \mathcal M'$, the
$\preceq$-morphism $f\pi$ $\preceq$-lifts to a
 $\preceq$-morphism $\tilde f: \mathcal Q \to \mathcal M$,
i.e., $ h\tilde f \succeq  f\pi.$ Hence $ h\tilde f \nu\succeq f\pi
\nu  \succeq f,$ so $\tilde f\nu$ $\preceq$-lifts $f$. This proves
that $\mathcal P$ is $\preceq$-projective.

One may prove the h-split case by the analogous argument.
\end{proof}


 \begin{prop}  \label{projspl2}
 Suppose $\mathcal Q$ is the $\preceq$-direct sum (resp.~ h-direct sum) of $\mathcal P_i$, with each $\mathcal P_i$ a $\preceq$-retract (resp.~ h-retract) of $\mathcal Q.$ If the $\mathcal P_i$
    are  $(\preceq,h)$-projective (resp.~h-projective) then $\mathcal Q$ is also
 $(\preceq,h)$-projective (resp.~h-projective).
 \end{prop}
 \begin{proof}
 We write $\nu_i: \mathcal P_i \to \mathcal Q$ and $\pi_i : \mathcal Q \to \mathcal P_i$ as in Definition~\ref{definition: direct sum}.
 For any   $\preceq$-onto  homomorphism $h: \mathcal M \to \mathcal
 M'$ and  $\preceq$-morphism $ f  : \mathcal Q \to \mathcal M',$
define the  $\preceq$-morphisms $ f_i = f \nu_i: \mathcal P_i \to
\mathcal M'$, which $\preceq$-lift to
 $\preceq$-morphisms $\tilde f_i: \mathcal P_i \to \mathcal M$,
i.e., $ h\tilde f_i \succeq f   \nu_i .$ But then $ h\tilde f_i
\pi_i \succeq f   \nu_i  \pi _i,$   so $h (\sum_i \tilde f_i  \pi _i
)\succeq  f \sum_i (\nu_i \pi_i) \succeq f,$ implying $\sum_i \tilde
f_i \pi_i$ $\preceq$-lifts $f.$ The same argument holds for
homomorphisms since a finite sum of homomorphisms is a homomorphism.
\end{proof}

The corresponding result for $\succeq$-projective is proved
analogously.

 \begin{prop}  \label{projspl2}
 Suppose $\mathcal Q$ is the $\succeq$-direct sum  of $\mathcal P_i$, with each $\mathcal P_i$ a $\succeq$-retract
   of $\mathcal Q.$ If the $\mathcal P_i$
    are  $ \succeq $-projective  then $\mathcal Q$ is also
 $ \succeq $-projective.
 \end{prop}
 \begin{proof}
 We write $\nu_i: \mathcal P_i \to \mathcal Q$ and $\pi_i : \mathcal Q \to \mathcal P_i$ as in Definition~\ref{definition: direct sum}.
 For any  $\succeq$-onto  homomorphism $h: \mathcal M \to \mathcal
 M'$ and  $\succeq$-morphism $ f  : \mathcal Q \to \mathcal M',$
define the  $\succeq$-morphisms $ f_i = f \nu_i: \mathcal P_i \to
\mathcal M'$, which $\succeq$-lift to
 $\succeq$-morphisms $\tilde f_i: \mathcal P_i \to \mathcal M$,
i.e., $ h\tilde f_i \preceq f   \nu_i .$ But then $ h\tilde f_i
\pi_i \succeq f   \nu_i  \pi _i,$   so $h (\sum_i \tilde f_i  \pi _i
)\preceq  \sum_i h \tilde f_i  \pi _i \preceq f \sum_i (\nu_i \pi_i)
\preceq f,$ implying $\sum_i \tilde f_i \pi_i$ $\succeq$-lifts $f.$
\end{proof}

\begin{prop}\label{retractlift}
If $\mathcal Q$ is $\preceq$-quasi-isomorphic to  $\mathcal P_1$ and
$\mathcal P_1$ is $(\preceq,h)$-projective 
, then $\mathcal Q$ is also
$(\preceq,h)$-projective. 
\end{prop}
 \begin{proof}
 The same proof as before. We take the $\preceq$-quasi-isomorphism $\pi: \mathcal Q \to \mathcal
 P_1$.  By $(\preceq,h)$-projectivity, there is a $\preceq$-retract
 $\nu: \mathcal
 P_1\to \mathcal Q$.
 For any   $\preceq$-onto  homomorphism $h: \mathcal M \to \mathcal
 M'$ any  $\preceq$-morphism $ f  : \mathcal Q \to \mathcal M' $
 $\preceq$-lifts to a
 homomorphism $\tilde f : \mathcal P_1 \to \mathcal M$,
i.e., $ h\tilde f  \succeq f   \nu  .$ But then $ h\tilde f  \pi
\succeq f   \nu   \pi  ,$   so in view of Lemma~\ref{kertriv}, $h (
\tilde f \pi )\succeq
 f,$ implying $\tilde f \pi$ $\preceq$-lifts $f.$ \end{proof}
%

 In \cite{IKR6} a stronger
version of projectivity is used in the tropical theory,
 studied intensively in \cite{KNZ}, namely,

\begin{defn} A systemic module    is  \textbf{strongly projective} if it is  a direct summand of a
free  systemic module.
\end{defn}

  An example was
given in \cite{IKR6} of a projective module that is
 not strongly projective.
\begin{rem}[cf.~{\cite[Proposition~17.14]{golan92}}] \label{partpro}
Every strongly projective systemic module  is  projective,
$\preceq$-projective, h-projective, and $\succeq$-projective, seen
by passing to the free systemic module  and appealing to
Lemma~\ref{lemma: free system is projective} and Lemma~\ref{lemma:
project direct sum}. The theory of strongly projective modules is
nice, but too restrictive for our purposes for homology.
\end{rem}

\begin{thm}\label{Sch2} If $\mathcal P_1 $ is
 $(\preceq,h)$-projective  with a  $\preceq$-onto homomorphism $\pi: \mathcal P
\longrightarrow \mathcal P_1 $ whose null-module kernel $\mathcal K$
 is
  $(\preceq,h)$-projective, then $\mathcal P $  also is
 $(\preceq,h)$-projective.
\end{thm}
\begin{proof}
We first lift the identity map of $\mathcal P_1$ to a
$\preceq$-retract $\nu: \mathcal P_1 \to \mathcal P$ of ${\pi}$, and
let $\pi_2 = \one_{\mathcal P}(-)\nu \pi$, which is
$\preceq$-idempotent by Lemma~\ref{lem: 3.14}(ii). Consider a
$\preceq$-onto homomorphism $h: \mathcal M \to \mathcal M'$. Then
for any $\preceq$-morphism $f: \mathcal P \to \mathcal M'$, we
 lift $f \nu : \mathcal P_1 \to \mathcal M'$ past $h$ to a  $\preceq$-morphism
 $\tilde f_1: \mathcal P_1 \to \mathcal M.$ Next, we lift $f |_{\mathcal K}: \mathcal K \to \mathcal M'$ to a $\preceq$-morphism
 $\tilde f_2: \mathcal K \to \mathcal M.$ We claim that for any $b \in \mathcal P$, $\pi_2(b) \in \mathcal K$. In fact, for any $b \in \mathcal P$,
 $\pi(b) (-) \pi \nu\pi (b) \succeq \pi(b) (-) \pi(b) \succeq \zero,$
 so
 \[
 \pi (\pi_2 (b)) = \pi (\one_\mathcal P (-) \nu \pi) (b)=\pi (b (-) \nu\pi(b))=\pi(b) (-) \pi \nu\pi (b) \succeq \pi(b) (-) \pi(b) \succeq \zero.
 \]
This implies that $\pi (\pi_2 (b)) \in (\mathcal P_1)_{\textrm{
Null}}$ and hence $\pi_2(b) \in \mathcal K$. Now, we define a
$\preceq$-morphism $\tilde{f}:\mathcal P \to \mathcal M$ as follows:
\[
\tilde f (b) =   \tilde f_1 (\pi  (b)) +   \tilde f_2 (\pi_2 (b)).
\]
Then $\tilde{f}$ is  well-defined since $\pi_2(b) \in \mathcal K$.
For any $b$ in $\mathcal P$ we have, using
Definition~\ref{mor}(iii,v),
  \begin{equation}\label{laststab}\begin{aligned} h\tilde f (b)& =   h\tilde
 f_1(\pi (b) ) +  h \tilde f_2(\one(-)\nu\pi) (b) 
 \\& \succeq   f(\nu \pi (b)) +  f    ((\one(-)\nu\pi)(b)) \succeq f ( b+\nu \pi (b) (-) \nu\pi(b))
\succeq  f(b),\end{aligned}\end{equation}
 proving $ h\tilde  f \succeq f,$ i.e. $\tilde  f$ $\preceq$-lifts $ f.$

\end{proof}
\subsubsection{$\preceq$-idempotent and $\preceq$-von Neumann regular
matrices}\label{vNR}$ $

Recall that an $m \times n$ matrix $A$ (with entries in a commutative ring) is said to be \emph{von Neumann regular} if there exists a matrix $B$ such that $A=ABA$. Classically, there is a well-known correspondence among von Neumann regularity, idempotency, and projectivity. In the tropical setting, as pointed out in \cite{IJK}, projectivity can be expressed in terms
of idempotent and von Neumann regular matrices.
%
%
%
%

 In what follows, we assume that all matrices have entries in a system $\mathcal A$ unless otherwise stated. We generalize the aforementioned correspondence to the $\preceq$-version.

 \begin{defn}
 We say $A \preceq B$ for $m \times n$ matrices $A = (a_{i,j}),\ B= (b_{i,j}),$
  if $ a_{i,j}\preceq b_{i,j}$ for all $i,j$.

  An  $n \times n$ matrix $A$ is
 \textbf{$\preceq$-idempotent} if $A\preceq  A^2 .$

 An  $m \times n$ matrix $A$ is
 \textbf{$\preceq$-von Neumann regular} if there is an  $n \times m$ matrix $B$
 for which $A\preceq  ABA .$

\end{defn}

\begin{prop}\label{kertriv1} Suppose $A$ is
 $\preceq$-idempotent. Then the module $A  \mathcal F$ is
$\preceq$-projective; in other words the column space of $A$ is a
$\preceq$-projective $\mathcal A$-submodule of
 $\mathcal F$, and symmetrically  the row space of $A$ is a
$\preceq$-projective $\mathcal A$-submodule of
 $\mathcal F$ .  \end{prop}
 \begin{proof} Define $\pi: \mathcal F \to A  \mathcal F$ by $\pi(v)
 = Av.$ Then $\pi \preceq \pi ^2$, and taking $\nu : A  \mathcal F
 \to   \mathcal F$ to be the identity, we have $\one \preceq \pi
 \nu$ on $ A  \mathcal F$, so we conclude by
 Proposition~\ref{projspl}. \end{proof}

 \begin{cor}\label{vNr} If $A \preceq  ABA ,$
then   $AB  \mathcal F$ is $\preceq$-projective.
 \end{cor}
 \begin{proof}  $AB$ is $\preceq$-idempotent, since $AB \preceq  (AB)^2 = (ABA)B .$
\end{proof}

\begin{example}\label{idemproj} Proposition \ref{kertriv1} gives us
an explicit way of obtaining new $\preceq$-projectives, via
$\preceq$-idempotent matrices. For example, if $(-)$ is of the first kind (which happens in the supertropical case, see \cite[Definition 2.22]{Row16} for the definition), and $A'$ is an
idempotent matrix, then
$$(I + A')^2 = I + (A')^\circ + (A')^2 \succeq I + A',$$
and this can be done in general.
\end{example}

 The analogous results hold for $\succeq$-idempotent and
$\succeq$-von Neumann regular. This raises the question of whether
$AB  \mathcal F = A  \mathcal F$ when $ A \preceq ABA $. Clearly $ A
\mathcal F \preceq ABA \mathcal F \subseteq AB \mathcal F \subseteq
A \mathcal F$, which often implies equality, but a thorough
discussion would take us too far afield here.

Trlifaj~\cite{Tr} has
considered the dual to Baer's criterion:

We say a systemic module  $\mathcal M$ is \textbf{finitely
$\preceq$-generated} (as a systemic module ) if it is
$\preceq$-generated by a finite set of cyclic systemic modules.
%
%
%
%

\begin{rem}[{As in \cite[p.~2]{Tr}}]\label{Baer3}
Suppose for any  $\preceq$-onto homomorphism $h: \mathcal M \to
\mathcal M'$ of systemic modules, with $\mathcal M $ cyclic, that
every $\preceq$-morphism $ f: \mathcal P \to \mathcal M'$
$\preceq$-lifts to a $\preceq$-morphism $\tilde f: \mathcal P
\to\mathcal M$. Then this condition holds for~$\mathcal M$
$\preceq$-finitely generated. (Indeed, write $\mathcal M \preceq
\sum_{i=1}^t \mathcal A a_i$ for $a_i \in \mathcal M$,  apply the
criterion for each~$\mathcal A a_i$, and add the $\preceq$-liftings,
i.e., $\tilde f (a) = \sum \tilde f_i(a)$).
\end{rem}

\cite[Lemma~2.1]{Tr} gives a countable counterexample to this
condition, and presents a readable and interesting account of the
dual Baer criterion in the classical case.

\subsection{The $\preceq$-dual basis lemma}$ $

Deore and~Pati \cite{DeP} proved a dual  basis lemma for projective
modules, and the same proof works for $\preceq$-projectives,
h--projectives, and $\succeq$-projectives.

\begin{prop}\label{Dualbas}
A module pseudo-system $(\mathcal P, \tT_\mathcal P,(-),\preceq)$
that is $\preceq$-generated by $\{p_i \in \mathcal P:i \in I\}$ is
$(\preceq,h)$-projective (resp.~h-projective) if and only if there
are $\preceq$-onto $\preceq$-morphisms (resp.~homomorphisms) $g_i :
\mathcal P \to \mathcal A$ such that for all $a\in \mathcal A$ we
have $a \preceq\sum g_i(a) p_i,$ where $g_i(a) = \zero$ for all but
finitely many $i$.
\end{prop}
\begin{proof}
The  assertion can be copied almost word for word from the standard
proof, for example from \cite[p.~493]{Row08}. We take the free
systemic module  $\mathcal F = (\mathcal A ^{(I)},\tT^{(I)}, (-),
\preceq )$ with base $\{ e_i : i \in I\}$, and the $\preceq$-onto
 homomorphism $f:\mathcal F \to \mathcal P$ given by $f(e_i) = p_i, \
\forall i\!\in\!I.$ Also we define the canonical projections $\pi _j
: \mathcal F \! \to  \! \mathcal A$ by $\pi _j (e_i) = \delta
_{ij}.$  Thus $c = \sum \pi _i(c)e_i$ for any $c \in \mathcal F$.

$(\Rightarrow)$ In view of Proposition~\ref{epicspl} (or
\ref{epicspl1}), $f$ is $\preceq$-split (resp.~h-split), so we take
a $\preceq$-morphism (resp.~ homomorphism) $g: \mathcal P \to
\mathcal F$ with $f g \succeq 1_{\mathcal P}.$  Put $g_i = \pi _i g:
\mathcal P \to \mathcal A$. Then any $a\in \mathcal P$ satisfies
$$\begin{aligned} a & \preceq fg (a)= f \Big(\sum _i \pi _i (g (a)) e_i\Big)  \\ &
= \sum_i f( g_i(a) e_i)= \sum_i  g_i(a) f(e_i) = \sum_i g_i (a)
p_i,\end{aligned}$$ as desired. If  $g$ is  a homomorphism then each
$g_i$ is  a homomorphism, seen by checking components.

$(\Leftarrow)$ Defining $g: \mathcal P \to \mathcal F$ by $g(a) =
\sum g_i(a)e_i,$ we have $$fg(a) = \sum g_i(a)f(e_i) = \sum
g_i(a)p_i\succeq   a.$$ Thus $fg \succeq 1_{\mathcal P},$ so
$\mathcal P$ is $(\preceq,h)$-projective, by Proposition~\ref{epicspl}.
When each of the $g_i$ is a homomorphism then clearly $g$ is a
homomorphism.

\end{proof}

\begin{prop}\label{Dualbas1}
Suppose a module pseudo-system $(\mathcal P, \tT_\mathcal
P,(-),\preceq)$ is generated by $\{p_i \in \mathcal P:i \in I\}$.
Then $\mathcal P$ is $(\succeq,h)$-projective if and only if there are
$\succeq$-onto $\succeq$-morphisms $g_i : \mathcal P \to \mathcal A$
such that for all $a\in \mathcal A$ we have $a \succeq\sum g_i(a)
p_i,$ where $g_i(a) = \zero$ for all but finitely many $i$.
\end{prop}
\begin{proof}
The analogous argument to the proof of Proposition~\ref{Dualbas}
works, where $\preceq$ and $ \succeq$ are interchanged.
\end{proof}

\section{Versions of Schanuel's Lemma}\label{projreso}$ $

We   turn to  systemic versions of Schanuel's Lemma which could
eventually relate to systemic projective dimension (which is work in
progress). In the classical case, one reduces to the following case:
Given two exact sequences $\mathcal K \to \mathcal P \overset {f}
\to \mathcal M$ and $\mathcal K'  \to \mathcal P' \overset {f'} \to
\mathcal M$, with $f,f'$ epic and $\mathcal P, \mathcal P'$
projective, one concludes that $\mathcal P \oplus \mathcal K' =
\mathcal P' \oplus \mathcal K.$ However, for modules over general
semirings, one cannot expect this to hold.


 The right notion of
exactness for modules over semirings is rather subtle.
Still, one can mimic the standard proof \cite[pp.~165--167]{La} of
Schanuel's Lemma for modules over rings, by considering our more
general version of splitting, and avoiding mixing submodules with
kernels of homomorphisms (which are congruences). To this end, we
introduce the following definition of congruence kernels.

\begin{defn}\label{congker}
Let $f:\mathcal M \to \mathcal N$ be a $\preceq$-morphism.
\begin{enumerate}\eroman
    \item
    The \textbf{N-congruence kernel} $\ker_N f$  of $f$ is defined to be the following set:
    \[
\ker_N f:=\{ (a_0,a_1) \in \mathcal M \times \mathcal M : f(a_0) =
f(a_1) \}.
    \]
 \item  The \textbf{$\preceq$-congruence kernel} $\ker_{N, \preceq} f$  of $f$ is defined to be the following set:
    \[
\ker_{N,\preceq} f:=\{ (a_0,a_1) \in \mathcal M \times \mathcal M : f(a_0) =
f(a_1) ,\quad f(a_0), f(a_1) \in \mathcal N_\textrm{Null} \}.
    \]
\end{enumerate}
\end{defn}

\begin{lem}
Let $f:\mathcal M \to \mathcal N$ be a homomorphism of systems. Then
$\ker_N f$
and $ \ker_{N,\preceq} f$ 
are submodules
of~$\mathcal M \times \mathcal M$. Also, $\ker_N f$ and  $
\ker_{N,\preceq} f$ are congruences on $\mathcal M$.
\end{lem}
\begin{proof}
This is clear.
\end{proof}


\begin{thm}[Semi-Schanuel]\label{trSh}
    Suppose we have $\preceq$-morphisms $\mathcal P_1  \overset {f_1} \longrightarrow \mathcal M$ and
    $\mathcal P_2  \overset {f_2} \longrightarrow \mathcal M$  with $f_1$ and $f_2$  onto. (We are
    not assuming that either $\mathcal P_i $ is projective.) Let
 $$\mathcal P = \{ (b_1, b_2)   :\quad b_i \in \mathcal P_i, \ f_1(b_1) =
        f_2(b_2)\},$$   a submodule of $\mathcal P_1 \oplus \mathcal P_2$,
        together the restriction $\pi_i^{\operatorname{res}}$ of
        the projection  $\pi_i: \mathcal P \to \mathcal P_i$ on the $i$ coordinate, for $i = 1,2$.
    \begin{enumerate}\eroman
        \item   $\pi_1^{\operatorname{res}}:\mathcal P \to \mathcal P_{1}$ is an onto homomorphism and
        and there is an
        onto homomorphism $$ \ker_N \pi_1^{\operatorname{res}} \to \ker_N  f_2, $$
        (This part is
        purely semiring-theoretic and does not require a system.)
            \item The maps $f_1\pi_1^{\operatorname{res}} ,f _2 \pi_2^{res}: \mathcal P \to \mathcal M$ are the same.
     \item
        In the systemic setting, $\pi_1^{\operatorname{res}}$ also induces
         $\preceq$-quasi-isomorphism
     \[
    \pi_{N,\preceq}:  \ker_{N, \preceq} \pi_1^{\operatorname{res}} \to \ker_{N, \preceq} f_2.\]

     \item In (iii), if $f_1$ also is null-monic, we have the following $\preceq$-quasi-isomorphism:
    \[
     \ker_{N, \preceq}  f_1\pi_1^{\operatorname{res}} \to \ker_{N, \preceq}  f_2.
     \]

 \item
 If $\mathcal P_1 $ is projective,   then it is a retract
 of $\mathcal P$ with respect to the projection $\pi_1:\mathcal P \to \mathcal
 P_1$.

 \item
 If $\mathcal P_1 $ is $\preceq$-projective, then it is a
 $\preceq$-retract
 of $\mathcal P$ with respect to the projection $\pi_1:\mathcal P \to \mathcal P_1$, and $\mathcal P$ is the $\preceq$-direct sum of
 $\mathcal P_1$ and $(\one_\mathcal P (-) \nu_1 \pi_1)(\mathcal P)$.

    \end{enumerate}
\end{thm}
\begin{proof}
Clearly $\mathcal P $ is a submodule of $\mathcal P_1 \oplus \mathcal P_2$. We modify the standard proof.

(i) We first prove that $\pi_1^{\textrm{res}}$ is an onto homomorphism. In fact, clearly $\pi_1$ is a homomorphism since it is a projection and hence its restriction $\pi_1^{\textrm{res}}$ is also a homomorphism. Now, since $f_2$ is onto, for any $b_1 \in \mathcal P_{1}$ there is $b_2 \in \mathcal P_2$ such that $f_1(b_1)=f_2(b_2),$ implying $(b_1,b_2)\in \mathcal P$. Hence $\pi_1$ restricts to an onto homomorphism $\pi_1^{\textrm{res}}:\mathcal P \to \mathcal P_{1}$. For the remaining part of (i), one can easily see that
    \[
    \ker_N \pi_1^{\operatorname{res}} \subseteq \{((b_1, b_2),(b_1', b_2')) : b_1=b_1'  \}
    \]
    But, by the definition of $\pi_1^{\operatorname{res}}$ and $\mathcal P$, we have that, for $((b_1, b_2),(b_1, b_2')) \in \ker_N\pi_1^{\operatorname{res}}$,
    \[
    f_2(b_2)
    = f_1(b_1) = f_2(b_2'),
    \]
    which means that $(b_2,b_2') \in \ker_N f_2.$ In other
    words,
    \[
    \ker_N \pi _1^{res} = \{ ((b_1, b_2),(b_1, b_2')): b_1 \in \mathcal P_{1},\ (b_2,b_2') \in \ker
    _{N} f_2
    \}.
    \]
    We define an onto homomorphism as follows:
    \[
    \pi:\ker _{N} \pi _1^{res} \to \ker_N f_2, \quad ((b_1, b_2),(b_1, b_2')) \mapsto ( b_2 , b_2') ,
    \]
    where $ b_1 \in \mathcal P_{1}.$ 



(ii) $f_1\pi_1^{\operatorname{res}} (b_1,b_2) = f_1(b_1) = f_2(b_2)
= f _2 \pi_2^{res} (b_1,b_2).$

 (iii) The proof for the  $\preceq$-quasi isomorphism $\ker_{N, \preceq} \pi_1^{\operatorname{res}} \to \ker_{N, \preceq} f_2$ is similar to the proof of (i). 
Notice that $ \ker_{N,\preceq} \pi_1^{\operatorname{res}} $ consists
of those pairs $((b_1, b_2),(b_1', b_2'))$ in $\mathcal P \times \mathcal P $ for which $b_1 = b_1'\in
\mathcal P_{1,\textrm{Null}}$. But then
\[
f_2(b_2) = f_1(b_1) = f_1(b_1') = f_2(b_2'),
\]
which are all in $\mathcal M_\textrm{Null}$ since
$f_1(b_1)\in\mathcal M_\textrm{Null}$, implying $(b_2,b_2') \in
\ker_{N,\preceq} f_2 $. We define the $\preceq$-morphism
$\pi_{N,\preceq}:\ker _{N,\preceq} \pi _1^{\textrm{res}} \to \ker_{N,\preceq}
f_2$ by
\[
((b_1, b_2),(b_1', b_2')) \mapsto ( b_2 , b_2').
\]

 Suppose that $(b_2,b_2') \in \ker_{N,\preceq} f_2$, i.e.,
$f_2(b_2) = f_2(b_2') \in \mathcal M_\textrm{Null}$. Since $f_1$ is
onto, we can find $b_1,b_1'$ in~$\mathcal P_1$ such that
$f_1(b_1)=f_2(b_2)$ and $f_1(b_1')=f_2(b_2')$. Clearly, for any
$(b_2,b_2') \in \ker_{N,\preceq} f_2$, we have that $((b_1,
b_2),(b_1, b_2')) \in \ker _{N,\preceq} \pi _1^{res}$ which shows
that $\pi_{N,\preceq}$ is onto. All it remains to show is that
$\pi_{N,\preceq}$ is null-monic. Suppose that
\[
\pi_{N,\preceq}((b_1, b_2),(b_1, b_2'))=(b_2,b_2') \in (\ker_{N,\preceq}
f_2)_{\textrm{Null}}.
\]
This means that $b_2,b_2' \in (\mathcal P_2)_{\textrm{Null}}$. It
follows that $((b_1, b_2),(b_1, b_2'))$ is an element of $(\ker
_{N,\preceq} \pi _1^{res})_{\textrm{Null}}$, showing that
$\pi_{N,\preceq}$ is null-monic by Lemma \ref{kertriv}. Thus
$\pi_{N,\preceq}$ is an  $\preceq$-quasi-isomorphism.

(iv) The proof for the second $\preceq$-quasi-isomorphism $\ker_{N,
\preceq} f_1\pi_1^{\operatorname{res}} \to \ker_{N, \preceq}  f_2$
is analogous. Slightly abusing notation, we define the following
$\preceq$-morphism:
\[
\pi_\preceq:\ker_{N, \preceq}  f_1\pi_1^{\operatorname{res}} \to
\ker_{N, \preceq}  f_2, \quad ((b_1,b_2),(b_1',b_2')) \mapsto
(b_2,b_2').
\]
One can easily see that $\pi_\preceq$ is null-monic from  exactly
the same argument as above, along with the hypothesis that $f_1$ is
null-monic. Now, suppose that $(b_2,b_2') \in \ker_{N, \preceq}
f_2$. In other words, we have that $f_2(b_2) = f_2(b_2')$ and
$f_2(b_2),f_2(b_2') \in \mathcal M_{\textrm{Null}}$. Again, since
$f_1$ is onto, we can find an element
$\alpha:=((b_1,b_2),(b_1',b_2')) \in \mathcal P \times \mathcal P$. 
We claim that $\alpha \in \ker_{N, \preceq}
f_1\pi_1^{\operatorname{res}}$; in this case,
$\pi_\preceq(\alpha)=(b_2,b_2')$, showing that $\pi_\preceq$ is
onto. In fact, we have
\[
f_1\pi_1^{\operatorname{res}}(b_1,b_2)=f_1(b_1)=f_2(b_2) =
f_2(b_2')=f_1(b_1')=f_1\pi_1^{\operatorname{res}}(b_1',b_2').
\]
Furthermore, since $f_2(b_2) \in \mathcal M_{\textrm{Null}}$, we
have that $f_1\pi_1^{\operatorname{res}}(b_1,b_2),
f_1\pi_1^{\operatorname{res}}(b_1',b_2') \in \mathcal
M_{\textrm{Null}}$, proving our claim.


(v) Since $\mathcal P_1$ is projective  and $\pi_1:\mathcal P \to
\mathcal P_1$ is onto, $\pi_1$ splits via $\nu_1$ with $\pi_1 \nu_1
= 1$.

(vi) Take $\nu_1: \mathcal P_1 \to \mathcal P$ to be a
$\preceq$-morphism $\preceq$-splitting $\pi_1$ via the identity map
on $\mathcal P_1$, and we can apply  Theorem \ref{splitdir} since
$\pi_1$ is a homomorphism.

\end{proof}

The proof of (iv) seems to require a rather strong hypothesis, which
it would be nice to be able to delete.

Here are some $\preceq$-versions.

\begin{lem}[Semi-Schanuel, onto $\preceq$-version]\label{trSh118} 
Given homomorphisms $\mathcal P_1  \overset {f_1} \longrightarrow
\mathcal M$ and
    $\mathcal P_2  \overset {f_2} \longrightarrow \mathcal M$  with~$f_2$
    onto,
    \begin{enumerate}\eroman
        \item
        There is a submodule $$\mathcal P = \{ (b_1, b_2)   : f_1(b_1) =
        f_2(b_2)\}$$ of $\mathcal P_1 \oplus \mathcal P_2$.  Let $\pi_i$ denote the
    projection to $\mathcal P_i$ on the $i$-th coordinate, and $\pi_i^{\operatorname{res}}$ is its restriction to $\mathcal P$.
Then  $\pi_1^{\operatorname{res}}$ is onto.

          \item
There is a natural   
     homomorphism $ \pi:\ker _{N} \pi _1^{res} \to  \mathcal P_{2} $
    via
    \[
   \quad ((b_1, b_2),(b_1, b_2')) \mapsto   b_2 (-) b_2'  ,
    \] which induces  a  $\preceq$-morphism $$ \ker_N \pi_1^{\operatorname{res}} \to \ker_{\Mod,\mathcal P_2} f_2. $$
     \item $\ker_{\Mod,\mathcal P}\pi_1^{\operatorname{res}} =   \{ (b_1, b_2) \in \mathcal P: \ b_1 \succeq \zero,\ b_2   \in
     \ker_{\Mod,\mathcal P}  f_2\}.$
     \item $f_1\pi_1^{\operatorname{res}} (b_1,b_2)  =    f _2
\pi_2^{res} (b_1,b_2).$

 \item
 If $\mathcal P_1 $ is h-projective, then it is a
 h-retract
 of $\mathcal P$ with respect to the projection $\pi_1:\mathcal P \to \mathcal P_1$, and $\mathcal P$ is the h-direct sum of
 $\mathcal P_1$ and $(\one_\mathcal P (-) \nu_1 \pi_1)(\mathcal P)$.


    \end{enumerate}
\end{lem}
\begin{proof}
   (i) The same argument in Theorem \ref{trSh} works.


    (ii) As in Theorem \ref{trSh}(i), we have that
    \[
    \ker_N \pi_1^{\operatorname{res}} \subseteq \{((b_1, b_2),(b_1', b_2')) : b_1=b_1'
    \}.
    \]



But, by the definition of $\pi_1^{\operatorname{res}}$ and $\mathcal P$, we see, for $((b_1, b_2),(b_1, b_2')) \in
  \ker_N\pi_1^{\operatorname{res}}$, that
    \[
f_1(b_1) = f_2(b_2)= f_2(b_2'),
    \]
$\ \ \ $ and hence
    \[
    \ker_N \pi _1^{res} = \{ ((b_1, b_2),(b_1, b_2')): b_1 \in \mathcal P_{1}
     ,f_2(b_2)=f_2(b_2')\}.
    \]
$\ \ \ $ Therefore, we obtain
\[
f_2( b_2) (-) f_2(b_2')=f_2 (b_2 (-) b_2') \succeq \zero.
\]

Hence ${\pi}(\ker _{N} \pi _1^{res} ) \subseteq \ker_{\Mod,\mathcal
P_2}f_2$.

(iii) $(\supseteq)$ is clear. Conversely, suppose $\pi_1^{\textrm{res}} ( b_1,b_2)
\succeq \zero.$ Then $b_1 \succeq \zero,$ implying $f_2 (b_2)
= f_1(b_1) \succeq \zero,$ i.e., $ b_2   \in
\ker_{\Mod,\mathcal P}  f_2.$

 (iv) $f_1\pi_1^{\operatorname{res}} (b_1,b_2) = f_1(b_1) = f_2(b_2) = f _2
\pi_2^{res} (b_1,b_2).$

(v) Take $\nu_1: \mathcal P_1 \to \mathcal P$ be the homomorphism
$h-$splitting $\pi_1$ via the identity map on $\mathcal P_1$, and we
can apply  Theorem \ref{splitdir}.

%
%
%
%
%
%
\end{proof}
The next result, although not symmetric, does not require the onto
hypothesis.

\begin{lem}[Semi-Schanuel, $\preceq$-onto $\preceq$-version]\label{trSh11} Take $\preceq =
\preceq_{\operatorname{Null}}.$
    Suppose we are given  homomorphisms $\mathcal P_1  \overset {f_1} \longrightarrow \mathcal M$ and
    $\mathcal P_2  \overset {f_2} \longrightarrow \mathcal M$  with $f_2$    $\preceq$-onto.
    \begin{enumerate}\eroman
        \item
        There is a submodule $$\mathcal P _\preceq = \{ (b_1, b_2)   : f_1(b_1) \preceq
        f_2(b_2)\}$$ of $\mathcal P_1 \oplus \mathcal P_2$.  Let $\pi_i$ denote the
    projection to $\mathcal P_i$ on the $i$-th coordinate, and $\pi_i^{\operatorname{res}}$ is its restriction to $\mathcal P_\preceq$.
Then  $\pi_1^{\operatorname{res}}$ is onto.

          \item
    There is a natural   homomorphism $ \ker _{N} \pi _1^{res} \to  \mathcal P_{2} $
    via
    \[
   \quad ((b_1, b_2),(b_1, b_2')) \mapsto   b_2 (-) b_2'  ,
    \]  which restricts to a homomorphism $ \pi:\ker_{N, \preceq}  \pi _1^{res}\to  \mathcal P_{2} $
    whose image is a subset of the following set $$\{  b_2(-)b_2' :  f(b_2) (-) f(b_2') \in
    \mathcal M^\circ\}.$$  Furthermore, we have $  \pi: \ker_{N, \preceq}\pi_1^{\operatorname{res}} \to \ker_{\Mod,\mathcal P_2} f_2.$
     \item $\ker_{\Mod,\mathcal P_\preceq}\pi_1^{\operatorname{res}} =   \{ (b_1, b_2) \in \mathcal P _\preceq: \ b_1 \succeq \zero,\ b_2   \in \ker_{\Mod,\mathcal P}  f_2\}.$
    \item $f_1\pi_1^{\operatorname{res}} (b_1,b_2)  \preceq    f _2
\pi_2^{res} (b_1,b_2).$
    \end{enumerate}
\end{lem}
\begin{proof}
   (i) The proof is similar to the previous cases. Clearly $\mathcal P_\preceq$ is a submodule of $\mathcal P_1 \oplus \mathcal P_2$.  Since $f_2$ is  $\preceq$-onto,
    for any $b_1 \in \mathcal P_{1}$
    there is $b_2 \in \mathcal P_2$ such that $f_1(b_1)\preceq f_2(b_2),$
    implying $(b_1,b_2)\in \mathcal P_\preceq$. Hence $\pi_1$ restricts to an
    onto homomorphism $\pi_1^{\operatorname{res}}:\mathcal P_\preceq\to \mathcal
    P_{1}$.

    (ii) As before, we have that
    \[
    \ker_N \pi_1^{\operatorname{res}} \subseteq \{((b_1, b_2),(b_1', b_2')) : b_1=b_1'
    \}.
    \]
Hence, the map
\begin{equation}\label{eqn: 1}
\ker _{N} \pi _1^{res} \to  \mathcal P_{2}, \quad ((b_1, b_2),(b_1, b_2')) \mapsto b_2 (-) b_2'
\end{equation}
is well-defined and clearly a homomorphism. Furthermore, we can restrict \eqref{eqn: 1} to a homomorphism
\[
\pi: \ker _{N,\preceq} \pi _1^{res} \to  \mathcal P_{2}, \quad ((b_1, b_2),(b_1, b_2')) \mapsto b_2 (-) b_2'
\]
since $\ker _{N,\preceq} \pi _1^{res} \subseteq \ker _{N} \pi _1^{res}$. Now, let $$X:=\{  b_2(-)b_2' :  f(b_2) (-) f(b_2') \in
\mathcal M^\circ\}.$$ For $\alpha:=((b_1, b_2),(b_1, b_2')) \in \mathcal P_\preceq \times \mathcal P_\preceq$, it is clear that $\alpha \in \ker _{N,\preceq} \pi _1^{res}$ if and only if $b_1 \succeq \zero$. Furthermore, as $\alpha \in \mathcal P_\preceq \times \mathcal P_\preceq$, in this case, we have that
\[
\zero \preceq f_1(b_1) \preceq f_2(b_2), \quad \zero \preceq f_1(b_1) \preceq f_2(b_2').
\]
So writing $f_2(b_2)=  f_1(b_1) + c $ and  $f_2(b_2')= f_1(b_1) +
{c'} $ for some $c,c' \succeq \zero,$ we obtain $$f_2( b_2) (-) f_2(b_2') =
f_1(b_1)^\circ + c (-){c'} \succeq \zero,$$ in particular, $f_2( b_2) (-) f_2(b_2') \in \mathcal M^\circ$. Therefore, we have that $\pi(\alpha) \in X$.

Since $f_2$ is a homomorphism, we have that $f(b_2) (-)
f(b_2')=f_2(b_2 (-) b_2')$ and hence the set $X$ becomes the
following set
\[
X=\{b_2(-)b_2': f_2(b_2(-)b_2') \in \mathcal M^\circ\},
\]
in particular, $X \subseteq \ker _{\textrm{Mod},\mathcal P_2} f_2$ and hence we have $\pi: \ker_{N, \preceq}\pi_1^{\operatorname{res}} \to \ker_{\Mod,\mathcal P_2} f_2$.





(iii) $(\supseteq)$ is clear. Conversely, suppose $\pi_1^{\textrm{res}} ( b_1,b_2)
\succeq \zero.$ Then $b_1 \succeq \zero,$ implying $f_2 (b_2)
\succeq f_1(b_1) \succeq \zero,$ i.e., $ b_2   \in
\ker_{\Mod,\mathcal P}  f_2.$

 (iv) $f_1\pi_1^{\operatorname{res}} (b_1,b_2) = f_1(b_1) \preceq f_2(b_2) = f _2
\pi_2^{res} (b_1,b_2).$


\end{proof}

 We also have the following $\preceq$ analogs of the
classical proof of Schanuel.


\begin{thm}[Semi-Schanuel, another $\preceq$-version]\label{trSh119}
Given a $\preceq$-morphism $\mathcal P   \overset {f }
\longrightarrow \mathcal M$ and
a homomorphism $\mathcal P'  \overset {f'} \longrightarrow \mathcal M'$  with $\mathcal P$ and $\mathcal P'$ $\preceq$-projective and $f$  $\preceq$-onto,
    and a $\preceq$-onto $\preceq$-morphism $\mu: \mathcal M \to \mathcal
    M'$,
let $\mathcal K =  \ker_{\Mod,\mathcal P}f$ and  $\mathcal K' =
\ker_{\Mod,\mathcal P'} f'.$ Then there is a $\preceq$-onto
$\preceq$-splitting $\preceq$-morphism $g: \mathcal K' \oplus
\mathcal P \to \mathcal P'$, with a $\preceq$-morphism $\Phi:
\mathcal K \to \ker_{\Mod,\mathcal K' \oplus \mathcal P}g $ which is
1:1 (as a set-map).
\end{thm}
\begin{proof} Lifting $\mu f$ to a $\preceq$-morphism $h: \mathcal P \to \mathcal P' $
satisfying $\mu f \preceq f'h,$ define $g: \mathcal K' \oplus
\mathcal P \to \mathcal P'$ by $g(b',b) = h(b)(-)b'.$

We first claim that $g$ is $\preceq$-onto. In fact, one may observe that for any $b' \in \mathcal P'$, there exists $b \in \mathcal P$ such that
\[
\mu f(b) \succeq f'(b').
\]
Indeed, let $c=f'(b')$. Since $\mu$ is $\preceq$-onto, there exists $x \in \mathcal M$ such that $c \preceq \mu(x)$. Moreover, since $f$ is $\preceq$-onto, we have $b \in \mathcal P$ such that $x \preceq f(b)$, in particular, we have that
\[
f'(b')=c \preceq \mu(x) \preceq  \mu f (b).
\]
Since $f'$ is a homomorphism, this implies that
\[
f'(h(b) (-)b') =f'(h(b)) (-) f'(b')\succeq   \mu f(b) (-) f'(b')
\succeq \zero.
\]
Therefore, we have that
\[
h(b) (-)b'\in \mathcal K'.
\]
Furthermore, we have that
\[
g(h(b)(-)b',b) =
h(b)(-)(h(b)(-)b') \succeq b',
\]
implying $g$ is $\preceq$-onto. Since $\mathcal P'$ is $\preceq$-projective, $g$ $\preceq$-splits (see, Remark \ref{remark: for later use}.)

For the last assertion that there is a $\preceq$-monic $\Phi:\mathcal K \to \ker_{\Mod,\mathcal K' \oplus
    \mathcal P }g$, take the map
\[
\Phi:\mathcal K \to \ker_{\Mod,\mathcal K' \oplus
    \mathcal P }g, \quad b \mapsto ( h(b),b).
\]
One can easily see that $\Phi$ is well-defined since $b \in \mathcal K$ implies that
\[
f'h (b) \succeq \mu f(b) \succeq \mu  (\zero) = \zero,
\]
showing that $h(b) \in \mathcal K'$, also
\[
g(h(b),b)=h(b) (-) h(b) \succeq \zero,
\]
showing that $(h(b),b) \in \ker_{\Mod,\mathcal K' \oplus
    \mathcal P }g$. Finally, it is clear that $\Phi$ is one-to-one as a set-map.

\end{proof}


 Theorem~\ref{trSh119} can sometimes be used in conjunction with  Theorem~\ref{Sch2}:

\begin{cor}\label{Sch29} In the notation of Theorem~\ref{trSh119}, if $\mathcal K
$ is $\preceq$-projective and the map $\Phi$ is the
$\preceq$-retract
    of a split $\preceq$-morphism $\ker_{\Mod,\mathcal K' \oplus \mathcal P}\to \mathcal K,$ then
$\mathcal K'$ also is $\preceq$-projective.
\end{cor}
\begin{proof}
$\ker_{\Mod,\mathcal K' }g$ is $\preceq$-projective, implying $
\mathcal K' \oplus \mathcal P $ is $\preceq$-projective, and thus  $
\mathcal K'   $ is $\preceq$-projective by~Lemma~\ref{lemma: project
direct sum}.
\end{proof}

\end{document}